\documentclass[11pt, a4paper]{article}
\usepackage{amsmath, amssymb}
\usepackage[margin=1in]{geometry}
\usepackage{amsthm}
\usepackage{comment}
\usepackage{color}
\usepackage{mathtools}
\usepackage{fancyhdr}
\usepackage[numbers]{natbib}
\usepackage{hyperref}

\newtheorem{theorem}{Theorem}
\newtheorem{prop}{Proposition}
\newtheorem{lemma}{Lemma}
\newtheorem{corollary}{Corollary}
\newtheorem{assumption}{Assumption}

\definecolor{editcol}{rgb}{0.0,0.5,0.0}

\definecolor{edittwocol}{rgb}{0.65,0.0,0.65}

\title{\textbf{The Generalized Excursion Coupling as the Limit of Concave Optimal Transport\thanks{A complete Lean~4 formalization of both main convergence
theorems is available at \protect \url{https://github.com/ykanoria/concave-OT-limit-lean} (tag \texttt{arxiv-v1}).}} 
}
\author{Yash Kanoria\\ Columbia University}

\begin{document}
\maketitle

\begin{abstract}
The Monge--Kantorovich problem with the Euclidean distance cost is degenerate, typically admitting infinitely many optimal plans. A unique optimal plan is selected by perturbing the distance to a strictly convex or increasing strictly concave cost and passing to the limit: the convex side gives the plan induced by the map which is monotone on each transport ray, while the concave side was understood only on the real line, where Juillet proved that power-cost optimizers converge to the so-called excursion coupling, which is induced by a map when the source is atomless. This paper extends the concave selection to $\mathbb{R}^n$ for mutually singular finite positive Borel measures $\mu$ and $\nu$ with equal total mass and finite first moments, assuming that $\mu\ll \mathcal L^n$: for a broad class of increasing strictly concave perturbations of the distance with a well-defined first-order profile,
the corresponding optimal maps converge in $\mu$-measure to the same intrinsic limit $t_\#$, independent of the perturbation family and its profile. The map $t_\#$ is obtained by disintegrating the transport along its maximal rays and applying the one-dimensional excursion coupling on each ray. The same limit is established for the power costs $\|x-y\|^{1-\varepsilon}$ under a finite $\|x\|\ln(1+\|x\|)$ moment condition, establishing a conjecture of Juillet.
\end{abstract}



\section{Introduction}
\label{sec:intro}

Let $\mu, \nu \in \mathcal{P}_1(\mathbb{R}^n)$ be probability measures with finite first moments.
 A map $t:\mathbb{R}^n\to\mathbb{R}^n$ is a \emph{transport map} from $\mu$ to $\nu$ if it pushes $\mu$ forward to $\nu$, i.e., $\mu(t^{-1}(B))=\nu(B)$ for every measurable $B\subseteq\mathbb{R}^n$. The Monge problem \cite{monge1781memoire} with Euclidean cost seeks a transport \emph{map} $t$ from $\mu$ to $\nu$ that minimizes the total transport cost
\begin{equation*}
 \quad \min \left\{ \int_{\mathbb{R}^n} \|x-t(x)\| \, d\mu : t \textup{ is a transport map from $\mu$ to $\nu$}\right\}\, ,
\end{equation*}
where $\|\cdot\|$ is the Euclidean norm.
This problem is non-linear in $t$, and hence may not have a solution, prompting the need for a relaxation. Let $\Pi(\mu,\nu)$ denote the set of couplings of $\mu$ and $\nu$, i.e., measures $\gamma \in \mathcal{P}(\mathbb{R}^n \times \mathbb{R}^n)$ with marginals $\mu$ and $\nu$: $\gamma(A\times\mathbb{R}^n)=\mu(A)$ and $\gamma(\mathbb{R}^n\times B)=\nu(B)$ for all measurable $A,B\subseteq\mathbb{R}^n$. The Monge--Kantorovich problem with the Euclidean cost is
\begin{equation*}
    (L^1) \quad \min \left\{ \int_{\mathbb{R}^n \times \mathbb{R}^n} \|x-y\| \, d\gamma : \gamma \in \Pi(\mu,\nu) \right\}\, .
\end{equation*}
In contrast to problems with costs which are strictly convex or strictly concave in the distance, the distance-cost problem is degenerate: its set $\Pi_1(\mu,\nu)$ of optimal plans is typically infinite. %
The transport set is partitioned into maximal oriented transport rays, along which an optimal plan can be disintegrated; see \cite[Theorems~6.1 \& 6.2]{ambrosio2003existence} and \cite[Propositions~18.5 \& 18.6]{maggi2023optimal}.
Singling out one optimal plan calls for a principled tie-breaking rule. We will be particularly interested in optimal plans induced by a map, which solve the Monge problem.

Canonical rules perturb the Euclidean cost to a strictly convex or increasing strictly concave function of the distance and let the perturbation vanish; the two directions select fundamentally different 
$L^1$-optimal plans. For each fixed perturbation, classical uniqueness results apply under suitable source regularity. If $\mu\ll\mathcal L^n$, a strictly convex cost such as $\|x-y\|^{1+\varepsilon}$
has a unique optimal plan induced by a map \cite[Theorem~1.2]{gangbo1996geometry}. For an increasing strictly concave function of the distance, all common mass is matched on the diagonal. If the residual
source is absolutely continuous with respect to $\mathcal L^n$---or, more generally, assigns zero mass to every $(n-1)$-rectifiable set---then the mutually singular residuals are coupled by a unique map
as shown by Pegon, Piazzoli, and Santambrogio \cite{pegon2013concave-general}, who refined the result of Gangbo and McCann \cite{gangbo1996geometry}.

Ambrosio and Pratelli \cite{ambrosio2003existence} (see also \cite{caffarelli2002constructing,trudinger2001monge}) showed that, for $\mu\ll\mathcal{L}^n$, the strictly convex costs $\|x-y\|^{1+\varepsilon}$
select, as $\varepsilon\to 0^+$, an optimal $L^1$ map that is \emph{monotone} along each transport ray. (Note that some source regularity, such as absolute continuity $\mu\ll\mathcal{L}^n$, is essential for the optimal plan to be induced by
a map: if $\mu$ is a Dirac mass and $\nu$ is not, no transport map exists at all \cite[p.~124]{ambrosio2003existence}.) In the opposite direction, Juillet proved on $\mathbb R$ that the optimal plans
for $|x-y|^q$ converge weakly, as $q\to 1^-$, to the excursion coupling \cite[Corollary~0.2]{juillet2020solution}. For every $q$ in the open interval $]0,1[$, this coupling is also the unique minimizer of $\int |x-y|
^q\,d\gamma$ among the $L^1$-optimal couplings; equivalently, it is characterized by concentration on a full-mass set which satisfies specific properties  \cite[Main Theorem and Definition~0.3]{juillet2020solution}. 
Its diagonal part matches the common mass, while its off-diagonal part couples the mutually singular residuals \cite[Lemma~2.3]
{juillet2020solution}. If the residual source is atomless, the off-diagonal coupling is induced by a map
\cite{juillet2020solution}. Thus fixed concave-cost optimizers were structurally understood in $\mathbb R^n$, but the canonical concave selection as one approaches the limit of (linear) Euclidean distance had been
identified only on the line.

Our main results carry the concave selection to $\mathbb{R}^n$. 
We treat the admissible increasing strictly concave families specified in Assumption~\ref{ass:concave-perturbation-family}, whose first-order profile $\phi(d)\coloneqq \lim_{\varepsilon\to 0^+}(\phi_\varepsilon(d)-d)/\varepsilon$ exists, is finite-valued, strictly concave, and satisfies $\phi(d)\ge -C(1+d)$ for some constant $C \in [0, \infty[$. 
One example is $c_\varepsilon(x,y)=\|x-y\|+\varepsilon\sqrt{\|x-y\|}$, with $\phi(d)=\sqrt{d}$. 
Since the common mass $\mu\wedge\nu$ is matched in place under any increasing strictly concave cost, we find it convenient to state our main result for mutually singular measures, with a corollary for arbitrary probability measures.
For mutually singular finite positive Borel measures $\mu$ and $\nu$ with equal positive total mass and finite first moments, assuming $\mu\ll\mathcal{L}^n$, the optimal plan is induced by a map $t_\varepsilon$; we show that these maps converge in $\mu$-measure to a single intrinsic limit $t_\#$ which is independent of the perturbation family and its first-order profile. We call $t_\#$ the \emph{generalized excursion-coupling (EC) map} (Theorem~\ref{thm:main-theorem}). Juillet's power costs $\|x-y\|^{1-\varepsilon}$ have profile $\phi(d)=-d\ln d$, which is not bounded below by $-C(1+d)$, so the aforementioned result does not apply. We extend our argument to establish the same intrinsic limit $t_\#$ for power costs under the stronger assumption of finite $\|x\|\ln(1+\|x\|)$ moments (Theorem~\ref{thm:main-theorem-power-costs}). The map $t_\#$ is built by disintegrating the transport into its maximal transport rays and applying Juillet's one-dimensional excursion coupling along each ray (Theorem~\ref{thm:6.1'}). This proves the conjecture stated by Juillet about the $\mathbb{R}^n$ limit \cite{juillet2020solution}.

Our method follows the $\Gamma$-convergence strategy of Ambrosio and Pratelli \cite{ambrosio2003existence}: any limit of the perturbed optima minimizes a \emph{secondary} variational problem---it minimizes $\int\phi(\|x-y\|)\,d\gamma$ over the $L^1$-optimal plans, for the first-order profile $\phi$ above. We show that concavity reverses the comparisons that drive the convex theory at each step of the argument, so the secondary problem selects the excursion coupling rather than the monotone map on each ray. We disintegrate an $L^1$-optimal plan along its maximal transport rays and, starting from Juillet's power-profile variational characterization \cite[Main Theorem]{juillet2020solution}, establish the corresponding raywise comparison for every admissible strictly concave family, not only the power family for which it was originally stated. The resulting excursion-coupling replacement depends only on the conditional marginals, not on the concave family, and the one-dimensional maps glue into the single measurable map $t_\#$.

Carrying this out in $\mathbb{R}^n$ requires disintegration into maximal transport rays and measurable selection over the ray space,  similar to that outlined in \cite[Section~6]{ambrosio2003existence} and \cite[Chapter~18]{maggi2023optimal}. This machinery could not simply be cited: checking the interface lemmas we rely on against the printed statements in both sources surfaced a few statement-level gaps (footnotes~\ref{fn:maximal-ray-gap}, \ref{fn:truncation-metric-gap}, \ref{fn:motionless-set-gap}), as well as gaps in the proofs (footnotes~\ref{fn:borel-ray-map-gap}, \ref{fn:measurable-gluing-gap}). The present paper proves the disintegration and selection steps independently (see Sections~\ref{sec:proof-of-main-theorem} and \ref{sec:proof-of-6.1'}), with every convention declared explicitly, and completes the argument.

We note that a different tie-breaking rule, \emph{entropic} regularization, has recently been studied in the same limiting spirit: Di Marino and Louet \cite{dimarino2017entropic} analyze the small-noise limit of entropically regularized distance-cost transport on the line, and Ley \cite{ley2025entropic} characterizes the limiting coupling under mutual singularity of the marginals. In general, these selections do not agree with the concave-limit map studied here.

{\bf Organization.} Section~\ref{sec:formulation} fixes the standing assumptions and states the main theorem for mutually singular measures, together with its general-measure corollary (Theorem~\ref{thm:main-theorem}). Section~\ref{sec:proof-of-main-theorem} reduces the concave limit to the secondary variational problem and identifies its unique solution. Section~\ref{sec:proof-of-6.1'} recalls the one-dimensional excursion coupling and proves the construction theorem for the generalized EC map $t_\#$. Section~\ref{sec:power-costs} extends the results to the power costs $\|x-y\|^{1-\varepsilon}$. We follow Juillet's interval convention: for $a<b$, the notations $]a,b[$, $[a,b[$, and $]a,b]$ denote the open, left-closed/right-open, and left-open/right-closed intervals, respectively.

\section{Problem Formulation and Main Result}
\label{sec:formulation}

We use the notation of Section~\ref{sec:intro}: the problem\footnote{Our results generalize to norms satisfying regularity and uniform convexity (condition (23) in \cite{ambrosio2003existence}).} $L^1$, the couplings $\Pi(\mu,\nu)$, and the Monge problem. For finite positive measures $\alpha,\beta$ with equal mass, $\Pi(\alpha,\beta)$ denotes the finite positive measures on $\mathbb{R}^n\times\mathbb{R}^n$ with marginals $\alpha$ and $\beta$, and we will use $\Pi_1(\alpha,\beta)$ to denote the minimizers of the Euclidean distance cost over this finite coupling class.

We study the limit of problems with increasing strictly concave costs $c_\epsilon(x,y) = \phi_\epsilon(\|x-y\|)$ for $\epsilon \in ]0,1[$, which approach the linear cost $\|x-y\|$ as $\epsilon \to 0^+$.

As shown by Gangbo and McCann \cite{gangbo1996geometry} and refined by Pegon et al \cite{pegon2013concave-general}, for any increasing strictly
concave cost function $c(x,y)=l(\|x-y\|)$ and equal mass measures $\mu$ and $\nu$ with $(\mu - \nu)_+ \ll \mathcal{L}^n$, the optimal plan decomposes
into two parts. The common mass $\mu \wedge \nu$ remains stationary on the diagonal, and the transport
problem reduces to optimally coupling the remaining, mutually singular parts $\mu_0 = (\mu - \nu)_+$ and $\nu_0 = (\nu - \mu)_+$. For
the latter measures, assuming that $\mu_0\ll\mathcal{L}^n$, the optimal plan under such concave costs is induced by
a deterministic transport map. This motivates us to state and prove our main theorem for the case of mutually singular measures with absolutely continuous source measure, with an easy corollary for arbitrary probability measures.

Juillet proved that one-dimensional optimizers for the power costs $|x-y|^q$, $q<1$, converge as $q\to 1^-$ to the \emph{excursion coupling} \cite[Corollary~0.2]{juillet2020solution}. His arch characterization asserts that this plan is concentrated on a full-mass relation $S\subset\mathbb{R}^2$ on which the following three properties hold for every two pairs $(x,y),(x',y')\in S$; it does not require these properties at every point of the topological support \cite[Main Theorem and Definition~0.3]{juillet2020solution}.
Let $[a,b]$ denote $[\min(a,b),\max(a,b)]$, and let $]a,b[$ denote its interior.
The properties are:
\begin{enumerate}
    \item Arches do not cross: either $[x,y]\cap[x',y'] = \emptyset$ or $[x,y]\cap[x',y'] = \{z\}$ for some $z \in \mathbb{R}$, or one of $[x,y]$ and $[x',y']$ is contained in the other.
    \item Arches do not connect: if $\min (|y-x|,|y'-x'|)>0$ then $ y \neq x'$.
    \item Nested arches have the same orientation: if $[x',y'] \subset ]x,y[$ we have $(y-x)(y'-x')\geq 0$.
\end{enumerate}
For mutually singular marginals, Juillet's defining relation is single-valued away from the atoms of the source \cite[Proposition~3.6]{juillet2020solution}. Consequently, if the source is atomless, disintegration over the first marginal and the standard measurable-kernel argument show that the excursion coupling is induced by a Borel \emph{excursion-coupling map}.

The extension to $\mathbb{R}^n$ uses the disintegration of an $L^1$-optimal plan, after restriction to a suitable plan-full contact set, into one-dimensional conditional plans on maximal oriented transport rays \cite[Theorems~6.2 and~9.1--9.4]{ambrosio2003existence}. Our canonical map between the mutually singular source and target measures, the {generalized excursion-coupling (EC) map} $t_\#$, applies the one-dimensional excursion-coupling map (Theorem~\ref{thm:EC-map-gives-a-plan-1D}) along almost every maximal transport ray; the conditional marginals are stochastically ordered and the conditional source is atomless. Theorem~\ref{thm:6.1'} gives the Borel raywise construction, and Theorem~\ref{thm:7.2'} proves that its induced plan is the common unique minimizer for every admissible strictly concave secondary profile.

\begin{assumption}[Concave perturbation family]
    \label{ass:concave-perturbation-family}
    Let $(\phi_\epsilon)_{\epsilon\in]0,1[}$ be a family of increasing strictly concave maps
$\phi_\epsilon:[0,\infty[\to[0,\infty[$ satisfying:
\begin{enumerate}
    \item For any fixed $d \ge 0$, the map $\epsilon \mapsto \phi_\epsilon(d)$ is convex, with $\phi_\epsilon(d) \to d$ as $\epsilon \to 0^+$.
    \item The right-derivative $\phi(d) \coloneqq  \lim_{\epsilon\to 0^+} \frac{\phi_\epsilon(d) - d}{\epsilon}$ exists and is a real valued strictly concave function on $[0, \infty[$, with lower bound $\phi(d) \geq -C(1+d)$ for some $C\in[0,\infty[$.
    \item $\phi_\epsilon(\cdot)$ satisfies sublinear growth, specifically, $\phi_\epsilon(d) \leq 2(1+d)$ for all $d \geq 0$ and $\epsilon \in ]0,1[$.
\end{enumerate}
\end{assumption}
A concrete family satisfying all these conditions is $\phi_\epsilon(d)=d+\epsilon \sqrt{d}$.

\begin{theorem}[Convergence of concave-cost optimizers to the generalized EC map.]
\label{thm:7.1'}
Let $\mu,\nu$ be mutually singular finite positive Borel measures on $\mathbb{R}^n$ with equal positive total mass and finite first moments, and assume $\mu\ll\mathcal{L}^n$. For $\epsilon \in ]0,1[$, consider the cost $c_{\epsilon}(x,y)=\phi_{\epsilon}(\|x-y\|)$
where $(\phi_\epsilon)_{\epsilon\in]0,1[}$ is a family of increasing strictly
concave maps satisfying
Assumption~\ref{ass:concave-perturbation-family} with first order profile $\phi(d)$.

There is a unique optimal transport plan in $\Pi(\mu,\nu)$ for $c_\epsilon$, and it is induced by a transport map $t_{\epsilon}$. As $\epsilon\to0^{+}$, the maps $t_{\epsilon}$ converge in $\mu$-measure to the 
generalized EC map $t_{\#}$, which depends only on $\mu$ and $\nu$ and not on the perturbation family $(\phi_\epsilon)_{\epsilon\in]0,1[}$ or its first-order profile.
\label{thm:main-theorem}
\end{theorem}

\begin{corollary}[General-measure form]
\label{cor:general-measure-main}
For $\epsilon \in ]0,1[$, consider the cost $c_{\epsilon}(x,y)=\phi_{\epsilon}(\|x-y\|)$
where $(\phi_\epsilon)_{\epsilon\in]0,1[}$ is a family of increasing strictly
concave maps satisfying
Assumption~\ref{ass:concave-perturbation-family}.
Let
$\mu,\nu\in\mathcal P_1(\mathbb R^n)$ and define
\[
\eta\coloneqq \mu\wedge\nu,\qquad \mu_0\coloneqq \mu-\eta,\qquad \nu_0\coloneqq \nu-\eta,\qquad
M_0\coloneqq \mu_0(\mathbb{R}^n)=\nu_0(\mathbb{R}^n).
\]
Assume $\mu_0\ll\mathcal{L}^n$. If $M_0=0$, then the unique optimal plan for $c_\epsilon$ is the diagonal plan on $\mu=\nu$. If $M_0>0$, then for each $\epsilon\in]0,1[$ the unique optimal plan for $c_\epsilon$ between $\mu$ and $\nu$ is
\[
\gamma_\epsilon=(\mathrm{Id},\mathrm{Id})_\#\eta+(\mathrm{Id},t_\epsilon)_\#\mu_0,
\]
where $t_\epsilon$ is the optimizer between the mutually singular residuals $\mu_0$ and $\nu_0$ from Theorem~\ref{thm:main-theorem}.
The residual maps $t_\epsilon$ converge to the residual generalized EC map $t_\#$ in $\mu_0$-measure. 
\end{corollary}

\begin{proof}
The residual measures $\mu_0$ and $\nu_0$ are mutually singular finite positive Borel measures with equal total mass $M_0$ and finite first moments, satisfying $\mu_0\ll\mathcal{L}^n$. If $M_0=0$, then $\mu=\nu=\eta$, and the diagonal plan is the unique optimizer because $c_\epsilon(x,y)$ is minimized only on the diagonal. Assume now that $M_0>0$ and fix $\epsilon\in]0,1[$. 
By Assumption~\ref{ass:concave-perturbation-family},
$\phi_\epsilon$ is finite-valued, nonnegative, increasing, and strictly
concave.
Since $\mu_0=(\mu-\nu)_+\ll\mathcal L^n$ and every
$(n-1)$-rectifiable set is $\mathcal L^n$-negligible, the rectifiability
hypothesis of \cite[Main Theorem]{pegon2013concave-general} is satisfied.
Applying that theorem to the full pair $(\mu,\nu)$ and the cost
$c_\epsilon(x,y)=\phi_\epsilon(\|x-y\|)$, the unique full optimizer has the form
\[
(\mathrm{Id},\mathrm{Id})_\#\eta+(\mathrm{Id},T_\epsilon)_\#\mu_0
\]
for a transport map $T_\epsilon$ from $\mu_0$ to $\nu_0$. The off-diagonal summand must be optimal for the residual problem between $\mu_0$ and $\nu_0$: otherwise, replacing $(\mathrm{Id},T_\epsilon)_\#\mu_0$ by a cheaper residual coupling while keeping $(\mathrm{Id},\mathrm{Id})_\#\eta$ fixed would give an admissible full coupling with strictly smaller $c_\epsilon$-cost. By Theorem~\ref{thm:main-theorem} applied to $(\mu_0,\nu_0)$, the residual optimizer is unique and equals $(\mathrm{Id},t_\epsilon)_\#\mu_0$. Hence $T_\epsilon=t_\epsilon$ $\mu_0$-a.e., which gives the displayed formula for $\gamma_\epsilon$. The convergence $t_\epsilon\to t_\#$ in $\mu_0$-measure is exactly the convergence statement in Theorem~\ref{thm:main-theorem} for the residual pair.
\end{proof}

We prove Theorem~\ref{thm:main-theorem} in Sections~\ref{sec:proof-of-main-theorem} and \ref{sec:proof-of-6.1'}.

A natural class of concave costs is the power costs, $c(x,y)= \|x-y \|^{1-\epsilon}$ for $\epsilon\in ]0,1[$. Theorem~\ref{thm:main-theorem} does not itself apply to these costs because the corresponding first-order profile $\phi(d)=-d\ln d$ is not bounded below by $-C(1+d)$. In Section~\ref{sec:power-costs}, we provide a version of our main theorem for power costs, under a slightly stronger assumption than finite first moments, namely, finite $\|x\|\ln(1+\|x\|)$ moments for the mutually singular measures (Theorem~\ref{thm:main-theorem-power-costs}), and obtain convergence to the same intrinsic map $t_\#$.

\section{Analysis leading to a proof of the Main Theorem}
\label{sec:proof-of-main-theorem}

The proof strategy follows that of Ambrosio and Pratelli \cite{ambrosio2003existence}, with significant changes to reflect that our costs are concave rather than convex.
Because the main analysis is formulated for mutually singular measures, every admissible plan has a diagonal-free full-mass supporting set.  Indeed, if $\mu\perp\nu$, choose a Borel set $A$ with $\mu(A^c)=0$ and $\nu(A)=0$; then every $\gamma\in\Pi(\mu,\nu)$ is concentrated on $A\times A^c$, which is disjoint from the diagonal. %

We first state the necessary geometric definitions. %
Given a transport plan concentrated on a diagonal-free set $\Gamma \subset (\mathbb{R}^n \times \mathbb{R}^n) \setminus \{(x,x) : x \in \mathbb{R}^n\}$:
\begin{itemize}
    \item A \textbf{transport ray} is an open oriented segment $]]x,y[[$ for $x \neq y$ and $(x,y) \in \Gamma$. The corresponding closed oriented segment is denoted $[[x,y]]$.
    \item The \textbf{transport set}, $T_\Gamma$, is the union of all transport rays: $T_\Gamma \coloneqq  \bigcup_{(x,y)\in\Gamma} ]]x,y[[$.
    \item The \textbf{left-transport set}, $T^l_\Gamma$, is the union of all segments $[[x,y[[$ for pairs $(x,y) \in \Gamma$.
    \item A \textbf{maximal transport ray} is an oriented non-empty
    open interval $S \subset \mathbb{R}^n$, with direction unit vector $v$, satisfying the following two conditions:
\begin{enumerate}
    \item[(a)] For any point $z \in S$, there exists a transport ray $]]x,y[[$ from $\Gamma$ such that $z \in ]]x,y[[$ and the direction vector $y-x$ is parallel to and co-oriented with $v$.
    \item[(b)] Any open interval $S' \supset S$ that also satisfies property (a) for the same direction $v$ must be equal to $S$.
\end{enumerate}
In contrast with the stated definition in \cite{ambrosio2003existence}, the above definition makes explicit that maximal transport rays are \emph{oriented} (they are maximal unions of collinear, consistently oriented and contiguous transport rays).\footnote{Our definition also requires $]]x,y[[$ to be collinear with $S$, which was not mentioned in the stated definition in \cite{ambrosio2003existence}. Collinearity is necessary to ensure that, under the no-crossing condition \eqref{eq:18} any point in $T_\Gamma$ is contained in a \emph{unique} maximal transport ray, so we (and they) do need to impose this requirement.\label{fn:maximal-ray-gap}}
\end{itemize}
Extending a transport ray within its oriented line shows that it is contained in a maximal transport ray.  Hence every point of $T_\Gamma$ belongs to at least one maximal transport ray.
We assume the following geometric constraint, referred to as the no-crossing condition in \cite{ambrosio2003existence}. This condition states that two closed transport rays with different orientations can meet only at a shared left endpoint or at a shared right endpoint.
\begin{equation} \label{eq:18}
    [[x,y]] \cap [[x',y']] \neq \emptyset \implies x=x' \text{ or } y=y' \quad \text{whenever } \tau(x,y) \neq \tau(x',y'),
\end{equation}
where $\tau(x,y) = (y-x)/\|y-x\|$ is the direction of the transport ray.
Under condition \eqref{eq:18}, two maximal transport rays containing the same point of $T_\Gamma$ have the same orientation and coincide; thus the maximal ray through such a point is unique.
Therefore $\Gamma$ induces a map $\pi_\Gamma: T_\Gamma \to \mathcal{S}_o(\mathbb{R}^n)$ which associates to any point the maximal transport ray containing it, where $\mathcal{S}_o(\mathbb{R}^n)$ denotes the space of open oriented intervals/rays. We also denote by $\tau_\Gamma: T_\Gamma \to \{v\in\mathbb{R}^n:\|v\|=1\}$ the map which gives the direction of the maximal transport ray containing the point.

\begin{lemma}[Borel maximal-ray map]
\label{lem:borel-maximal-ray-map}
Let
\[
    \Gamma\subset
    (\mathbb R^n\times\mathbb R^n)
    \setminus\{(x,x):x\in\mathbb R^n\}
\]
be $\sigma$-compact and satisfy condition~\eqref{eq:18}. Then $T_\Gamma$
is Borel and
\[
    \pi_\Gamma:T_\Gamma\longrightarrow\mathcal S_o(\mathbb R^n)
\]
is Borel.
\end{lemma}

\begin{proof}
Write $\overline{\mathbb R}=[-\infty,+\infty]$, with its order topology, and
introduce the auxiliary code space
\[
    \mathcal R_{\mathrm{code}}
    \coloneqq
    \left\{(a,v,\alpha,\beta):
    \begin{array}{l}
        a\in\mathbb R^n,\quad v\in\mathbb S^{n-1},\quad
        \langle a,v\rangle=0,\\
        \alpha,\beta\in\overline{\mathbb R},\quad \alpha<\beta
    \end{array}
    \right\}.
\]
It is an open subset of a closed subspace of
\[
    \mathbb R^n\times\mathbb S^{n-1}
    \times\overline{\mathbb R}^{\,2},
\]
so it is locally compact Polish and, in particular, standard Borel. The code
$C=(a,v,\alpha,\beta)$ represents the oriented open ray
\[
    C^\circ=\{a+tv:\alpha<t<\beta\}.
\]
Every nondegenerate oriented open interval, half-line, or affine line has a
unique such code.

We first verify that the decoding map
\[
    J_o:\mathcal R_{\mathrm{code}}\longrightarrow
    \mathcal S_o(\mathbb R^n),
    \qquad J_o(C)\coloneqq C^\circ,
\]
is Borel. For $S\in\mathcal S_o(\mathbb R^n)$ and $R\in\mathbb N$, let
$x_R(S),y_R(S)$ be the endpoints, in oriented order, of
$\overline S\cap\overline B_R$; assign the fixed pair $(0,0)$ when the
intersection is empty, and repeat the endpoint when it is a singleton.
Following \cite[Definition~6.1]{ambrosio2003existence}, define
\begin{align}
    d_{\mathrm{AP}}(S,S')
    \coloneqq
    \sum_{R=1}^{\infty}2^{-R}
    \frac{\|x_R(S)-x_R(S')\|+\|y_R(S)-y_R(S')\|}
    {1+\|x_R(S)-x_R(S')\|+\|y_R(S)-y_R(S')\|}.
    \label{eq:AP-metric}
\end{align}
This is the metric on $\mathcal S_o(\mathbb R^n)$ induced by the closure
correspondence. The oriented-order, empty-intersection, and singleton
conventions above make explicit the cases left implicit in the displayed
formula of Ambrosio and Pratelli.\footnote{Both \cite[Definition~6.1]{ambrosio2003existence} and \cite[p.~221]{maggi2023optimal} leave the endpoint ordering unstated; the natural-looking fix of ordering by the unordered pair alone makes $d_{\mathrm{AP}}$ unable to distinguish a segment from its reversal. Our oriented-order convention avoids this and is needed for the Borel argument in Lemma~\ref{lem:measurable-raywise-ec}.\label{fn:truncation-metric-gap}}

For $C=(a,v,\alpha,\beta)$ and $R\in\mathbb N$, on $\|a\|\le R$ put
\[
    \rho_R(C)\coloneqq\sqrt{R^2-\|a\|^2},\qquad
    L_R(C)\coloneqq\max\{\alpha,-\rho_R(C)\},\qquad
    U_R(C)\coloneqq\min\{\beta,\rho_R(C)\}.
\]
The closure of $C^\circ$ meets $\overline B_R$ exactly when
\[
    \|a\|\le R\qquad\hbox{and}\qquad L_R(C)\le U_R(C).
\]
On this Borel set the ordered endpoints of $J_o(C)$ are
\[
    x_R(J_o(C))=a+L_R(C)v,\qquad
    y_R(J_o(C))=a+U_R(C)v,
\]
and off this set both endpoints are $0$. Thus
$C\mapsto\bigl(x_R(J_o(C)),y_R(J_o(C))\bigr)$ is Borel. For each fixed
$S\in\mathcal S_o(\mathbb R^n)$,
the function
\[
    C\longmapsto d_{\mathrm{AP}}(J_o(C),S)
\]
is therefore a countable sum of Borel functions and is Borel. Since
$\mathcal S_o(\mathbb R^n)$ is separable, balls with centers in a countable
dense set and positive rational radii form a countable base. Their preimages
under $J_o$ are Borel, which proves that $J_o$ is Borel.

We now construct a Borel code for the maximal ray through each point. Let
\[
    \mathcal L
    \coloneqq
    \{(v,a)\in\mathbb S^{n-1}\times\mathbb R^n:
      \langle a,v\rangle=0\}.
\]
For $q=(x,y)\in\Gamma$, define
\[
    v(q)\coloneqq\frac{y-x}{\|y-x\|},\qquad
    a(q)\coloneqq x-\langle x,v(q)\rangle v(q),
\]
and
\[
    \ell(q)\coloneqq\langle x,v(q)\rangle,\qquad
    u(q)\coloneqq\langle y,v(q)\rangle.
\]
Then
\[
    ]]x,y[[
    =
    \{a(q)+tv(q):\ell(q)<t<u(q)\}.
\]
Define the oriented incidence set
\[
    \mathcal A
    \coloneqq
    \left\{
      \bigl(v(q),a(q),(1-s)\ell(q)+su(q)\bigr):
      q\in\Gamma,\ 0<s<1
    \right\}
    \subset\mathcal L\times\mathbb R.
\]
The defining map from $\Gamma\times]0,1[$ is continuous. Its domain is
$\sigma$-compact, so $\mathcal A$ is a $K_\sigma$ subset of the Polish
space $\mathcal L\times\mathbb R$ and is standard Borel. For each $(v,a)$,
the section
\[
    \mathcal A_{v,a}\coloneqq\{t:(v,a,t)\in\mathcal A\}
\]
is open, being a union of open intervals.

Let
\[
    e:\mathcal A\longrightarrow\mathbb R^n,\qquad
    e(v,a,t)\coloneqq a+tv.
\]
Its image is exactly $T_\Gamma$. We claim that $e$ is injective. Suppose
that two transport segments $]]x,y[[$ and $]]x',y'[[$ contain the same
point $z$ in their relative interiors. If their normalized directions
differ, condition~\eqref{eq:18} gives $x=x'$ or $y=y'$. In the first case,
the vectors from the common source to $z$ show that the normalized
directions are equal; in the second, the vectors from $z$ to the common
target give the same conclusion. Hence all transport segments containing
$z$ in their relative interiors have the same normalized direction. The
equalities
\[
    a=z-\langle z,v\rangle v,\qquad t=\langle z,v\rangle
\]
then give equality of the remaining incidence coordinates, proving
injectivity.

Since $\mathcal A$ is $K_\sigma$ and $e$ is continuous,
\[
    T_\Gamma=e(\mathcal A)
\]
is $K_\sigma$, hence Borel. The Lusin--Souslin theorem
\cite[Theorem~15.1]{kechris1995classical} now shows that
\[
    \xi\coloneqq(e|_{\mathcal A})^{-1}:T_\Gamma\longrightarrow\mathcal A
\]
is Borel. Write
\[
    \xi(z)=(v_z,a_z,t_z).
\]

It remains to select the connected component of
$\mathcal A_{v_z,a_z}$ containing $t_z$ in a Borel way. For
$\xi=(v,a,t)\in\mathcal A$ and $s\in\mathbb R$, let
$\operatorname{Conn}(\xi,s)$ mean that $s$ lies in the same connected
component of $\mathcal A_{v,a}$ as $t$. Equivalently,
\[
    [\min\{s,t\},\max\{s,t\}]\subset\mathcal A_{v,a}.
\]
In the ambient space $\mathcal A\times\mathbb R^2$, define
\[
    \mathcal D
    \coloneqq
    \left\{((v,a,t),s,w)\in\mathcal A\times\mathbb R^2:
        \min\{s,t\}\le w\le\max\{s,t\},\quad
        (v,a,w)\notin\mathcal A
    \right\}.
\]
This set is Borel, and its section over $(\xi,s)$ in the $w$-coordinate is
\[
    [\min\{s,t\},\max\{s,t\}]\setminus\mathcal A_{v,a},
\]
which is compact because $\mathcal A_{v,a}$ is open. The
Arsenin--Kunugui theorem
\cite[Theorem~18.18]{kechris1995classical} therefore shows that
\[
    P\coloneqq\operatorname{proj}_{\mathcal A\times\mathbb R}\mathcal D
\]
is Borel. Consequently
\[
    \operatorname{Conn}(\xi,s)
    \quad\Longleftrightarrow\quad
    (\xi,s)\notin P
\]
is a Borel relation.

Define
\[
    \alpha(\xi)\coloneqq\inf\{s:\operatorname{Conn}(\xi,s)\},\qquad
    \beta(\xi)\coloneqq\sup\{s:\operatorname{Conn}(\xi,s)\}.
\]
The set in these formulas is the nonempty open component containing $t$.
For every $c\in\mathbb R$, density of $\mathbb Q$ gives
\[
    \{\xi:\alpha(\xi)<c\}
    =
    \bigcup_{\substack{p\in\mathbb Q\\p<c}}
    \{\xi:\operatorname{Conn}(\xi,p)\},
\]
and
\[
    \{\xi:\beta(\xi)>c\}
    =
    \bigcup_{\substack{p\in\mathbb Q\\p>c}}
    \{\xi:\operatorname{Conn}(\xi,p)\}.
\]
Thus $\alpha$ and $\beta$ are Borel
$\overline{\mathbb R}$-valued functions, including when an endpoint is
infinite.

For $z\in T_\Gamma$, put
\[
    I_z\coloneqq
    \bigl(\alpha(\xi(z)),\beta(\xi(z))\bigr),\qquad
    S_z
    \coloneqq
    \{a_z+sv_z:s\in I_z\}.
\]
By construction, $I_z$ is the connected component of
$\mathcal A_{v_z,a_z}$ containing $t_z$. Hence every point of $S_z$ lies
in a transport ray co-oriented with $v_z$, so $S_z$ satisfies property~(a).
Any strictly larger open interval with the same direction and property~(a)
would yield a connected subset of $\mathcal A_{v_z,a_z}$ properly
containing $I_z$, which is impossible. Thus $S_z$ is maximal.

If $R$ is any maximal transport ray containing $z$, property~(a) supplies
a transport ray through $z$ co-oriented with $R$. Injectivity of $e$ forces
its incidence coordinates to equal $\xi(z)$, so $R$ has direction $v_z$
and supporting line $a_z+\mathbb R v_z$. Property~(a) then makes the
coordinate interval of $R$ a connected subset of
$\mathcal A_{v_z,a_z}$ containing $t_z$. Hence $R\subset S_z$, and
maximality of $R$ gives $R=S_z$. Consequently $S_z=\pi_\Gamma(z)$.

Since $I_z$ is nonempty, $\alpha(\xi(z))<\beta(\xi(z))$. Its auxiliary
code is
\[
    \widehat\pi_\Gamma(z)
    \coloneqq
    \bigl(a_z,v_z,\alpha(\xi(z)),\beta(\xi(z))\bigr)
    \in\mathcal R_{\mathrm{code}}.
\]
Every coordinate is Borel, so
$\widehat\pi_\Gamma:T_\Gamma\longrightarrow
    \mathcal R_{\mathrm{code}}$
is Borel. Finally,
$\pi_\Gamma=J_o\circ\widehat\pi_\Gamma$
is Borel.
\end{proof}

For comparison, related Borel ray-map statements appear in
\cite[Lemma~6.1]{ambrosio2003existence} and
\cite[Proposition~18.5]{maggi2023optimal}. Maggi's proposition assumes
that the generating set is closed, whereas
Lemma~\ref{lem:borel-maximal-ray-map} treats the present definition for
$\sigma$-compact $\Gamma$.

For a plan $\gamma$ concentrated on such a (diagonal-free, no-crossing)
set $\Gamma$, we denote by
$r:\Gamma\to\mathcal{S}_c(\mathbb{R}^n)$ the pair-to-ray map which sends
$(x,y)\in\Gamma$ to the closure of the unique oriented maximal transport ray
containing $]]x,y[[$, where $\mathcal S_c(\mathbb R^n)$ denotes the space of
closed oriented intervals/rays. 
Let
$\operatorname{cl}:\mathcal S_o(\mathbb R^n)\to
\mathcal S_c(\mathbb R^n)$ denote the closure map. For every
$(x,y)\in\Gamma$,
\begin{equation}
    r(x,y)
    =
    \operatorname{cl}\!\left(
        \pi_\Gamma\!\left(\frac{x+y}{2}\right)
    \right).
    \label{eq:pair-to-ray-midpoint}
\end{equation}
Indeed, $(x+y)/2\in]]x,y[[$.  The maximal ray containing
$]]x,y[[$ therefore contains this midpoint and, by uniqueness of the maximal
ray through a point of $T_\Gamma$, equals
$\pi_\Gamma((x+y)/2)$.  The midpoint representation avoids extending
$\pi_\Gamma$ to ray endpoints, since the midpoint always belongs to
$T_\Gamma$.
For given $\gamma$ and $\Gamma$, we write
$\sigma\coloneqq r_\#\gamma$ and denote the corresponding disintegration into closed
maximal transport rays by $\gamma=\gamma_C\otimes\sigma$. Finally, set
\[
    \mu_C \coloneqq  (\mathrm{proj}_1)_\#\gamma_C\, ,
    \qquad
    \nu_C \coloneqq  (\mathrm{proj}_2)_\#\gamma_C \, .
\]

For a cost function $c:\mathbb{R}^n\times\mathbb{R}^n\to\mathbb{R}$, a set
$\Gamma\subset\mathbb{R}^n\times\mathbb{R}^n$ is called
\emph{$c$-cyclically monotone} if, for every finite family
$(x_i,y_i)_{i=1}^N\subset\Gamma$ and every permutation $\tau$ of
$\{1,\ldots,N\}$,
\[
    \sum_{i=1}^N c(x_i,y_i)
    \le
    \sum_{i=1}^N c(x_i,y_{\tau(i)}).
\]

\begin{theorem}[Generalized Excursion-Coupled Map Construction for given $\Gamma$]
\label{thm:6.1'}
Let $\mu, \nu$ be mutually singular finite positive Borel measures on $\mathbb{R}^n$ with equal positive total mass and finite first moments.
Let $\gamma \in \Pi(\mu, \nu)$ be a finite transport plan concentrated on a $\sigma$-compact diagonal-free set %
$\Gamma$. Assume that the no-crossing condition \eqref{eq:18} holds and that:
\begin{enumerate}
    \item[(i)] $\mu$ is absolutely continuous with respect to $\mathcal{L}^n$;
    \item[(ii)] %
    $T^l_\Gamma \setminus T_\Gamma$
    is $\mu$-negligible;
    \item[(iii)] there exists a $\mu$-negligible set $N \subset T_\Gamma$ and an increasing sequence of compact sets $K_h$ such that $\tau_\Gamma |_{K_h}$ is a Lipschitz map and the union of $K_h$ is $T_\Gamma \setminus N$.
\end{enumerate}
Then there exists a transport plan $\gamma_\# \in \Pi(\mu, \nu)$ such that:
\begin{enumerate}
    \item[(a)] $\gamma_\#$ is induced by a transport map $t_\#$, and this plan $\gamma_\#$ is optimal for the cost $c(x,y)=\|x-y\|$ whenever $\Gamma$ is $c$-cyclically monotone.
    \item[(b)] $\gamma_\#$ is concentrated on the set of pairs $(x, y)$ such that %
    $[[x, y]]$ is contained in the closure of a maximal transport ray of $\Gamma$; %
    for $\sigma$-a.e. ray $C$, the underlying map $t_\#$ agrees $\mu_C$-a.e. with the one-dimensional excursion-coupling map defined in Theorem~\ref{thm:EC-map-gives-a-plan-1D} from $\mu_C$ to $\nu_C$ on $C$.
    \item[(c)] For any \textbf{concave} function $\phi : [0, +\infty[ \to \mathbb{R}$ satisfying $\phi(d)\ge -K_\phi(1+d)$ for all $d\ge0$ and some $K_\phi \in[0,\infty[$, %
    we have
    $$\int_{\mathbb{R}^n \times \mathbb{R}^n} \phi(\|x - y\|) \, d\gamma_\# \leq \int_{\mathbb{R}^n \times \mathbb{R}^n} \phi(\|x - y\|) \, d\gamma.$$
    If $\phi$ is strictly concave, the inequality above is strict unless $\gamma = \gamma_\#$.
\end{enumerate}
Moreover, for fixed $\mu,\nu,\Gamma$, the plan
$\gamma_\#^\Gamma$ depends only on $\mu,\nu,\Gamma$, and not on the
particular input plan $\gamma$ concentrated on $\Gamma$. In conclusions
(a)--(c), $\gamma_\#$ denotes this canonical plan $\gamma_\#^\Gamma$.
\end{theorem}

We prove Theorem~\ref{thm:6.1'} in the next section.

We will establish a regularity result similar to \cite[Theorem 6.2]{ambrosio2003existence} for the contact set of the linear-cost optimizers, and use it in conjunction with Theorem~\ref{thm:6.1'} to show uniqueness of the limit transport. (For a modern textbook treatment, including a detailed proof of the underlying countable-Lipschitz regularity of Kantorovich potentials, we refer the reader to \cite[Theorem~18.9]{maggi2023optimal}.)

The following is a standard consequence of Kantorovich duality for linear costs \citep[see, e.g.,][Theorem~3.17]{maggi2023optimal}.
\begin{lemma}[Contact set for the linear-cost optimizers]
    \label{lem:linear-cost-contact-set}
    Let $\mu,\nu$ be finite positive Borel measures on $\mathbb R^n$ with equal mass
    and finite first moments. Then there is a 1-Lipschitz function
    $u:\mathbb R^n\to\mathbb R$ such that, with
    \[
    \Gamma_u\coloneqq \{(x,y):\|x-y\|=u(x)-u(y)\},
    \]
    one has that the set of optimal plans for the Euclidean cost is given by
    \[
    \Pi_1(\mu,\nu)
    =
    \{\eta\in\Pi(\mu,\nu):\eta(\Gamma_u)=\mu(\mathbb R^n)\}.
    \]
    Moreover $\Gamma_u$ is $c$-cyclically monotone for $c(x,y)=\|x-y\|$.
\end{lemma}

\begin{proof}[Proof of Lemma~\ref{lem:linear-cost-contact-set}]
    If $\mu$ and $\nu$ have mass zero, the statement is trivial. Otherwise,
    put $M\coloneqq \mu(\mathbb R^n)=\nu(\mathbb R^n)>0$ and
    $\bar\mu\coloneqq M^{-1}\mu$, $\bar\nu\coloneqq M^{-1}\nu$. The correspondence
    $\eta\mapsto\bar\eta\coloneqq M^{-1}\eta$ is a bijection from
    $\Pi(\mu,\nu)$ to $\Pi(\bar\mu,\bar\nu)$, and for every nonnegative
    Borel cost $c$,
    \[
        \int c\,d\eta=M\int c\,d\bar\eta .
    \]
    Kantorovich duality for the probability measures $\bar\mu,\bar\nu$ therefore
    gives a 1-Lipschitz function $u$ which, after multiplying the dual identity
    by $M$, satisfies
    \[
        \min_{\eta\in\Pi(\mu,\nu)}\int \|x-y\|\,d\eta
        =
        \int u\,d\mu-\int u\,d\nu .
    \]
    For every $\eta\in\Pi(\mu,\nu)$,
    \[
        \int \|x-y\|\,d\eta
        \ge
        \int (u(x)-u(y))\,d\eta
        =
        \int u\,d\mu-\int u\,d\nu \, ,
    \]
    using that $u$ is 1-Lipschitz.
    Thus $\eta$ is optimal if and only if the nonnegative gap
    $\|x-y\|-u(x)+u(y)$ has zero $\eta$-integral, equivalently
    $\eta(\Gamma_u)=\mu(\mathbb R^n)$.

    It remains to check cyclic monotonicity. If
    $(x_i,y_i)_{i=1}^N\subset\Gamma_u$ and $\tau$ is a permutation, then
    \[
        \sum_i \|x_i-y_i\|
        =
        \sum_i (u(x_i)-u(y_i))
        =
        \sum_i (u(x_i)-u(y_{\tau(i)}))
        \le
        \sum_i \|x_i-y_{\tau(i)}\|,
    \]
    where the last inequality uses that $u$ is 1-Lipschitz. Hence $\Gamma_u$ is
    $c$-cyclically monotone for $c(x,y)=\|x-y\|$.
\end{proof}

The following theorem establishes that the linear-cost contact set selected in
Lemma~\ref{lem:linear-cost-contact-set} has the regularity needed for
Theorem~\ref{thm:6.1'}.

\begin{theorem}[$\Gamma_u$ is sufficiently regular]
\label{thm:6.2'}
Let $\mu,\nu$ be mutually singular finite positive Borel measures on
$\mathbb{R}^n$ with equal positive total mass and finite first moments, and
assume $\mu\ll\mathcal{L}^n$. Let $u$ be the 1-Lipschitz Kantorovich potential
given by Lemma~\ref{lem:linear-cost-contact-set} for the pair $(\mu,\nu)$, and
let $\Gamma_u$ be the corresponding contact set.
If $\Gamma$ is any $\sigma$-compact diagonal-free subset of $\Gamma_u$,
then the no-crossing condition~\eqref{eq:18} holds, and furthermore:
\begin{enumerate}
    \item[(i)] For any maximal transport ray $S$, we have
    $u(x')-u(y')=\|x'-y'\|$ whenever $x',y'\in S$ and $x'\le y'$
    with respect to the orientation of the ray.
    \item[(ii)] The set $T_\Gamma^l\setminus T_\Gamma$ is
    $\mathcal{L}^n$-negligible.
    \item[(iii)] Condition (iii) in Theorem~\ref{thm:6.1'} holds.
\end{enumerate}
\end{theorem}

\begin{proof}
    We write the proof in our oriented, diagonal-free notation. Since
    $\Gamma\subset\Gamma_u$, every transport pair $(x,y)\in\Gamma$ satisfies
    $u(x)-u(y)=\|x-y\|$. Hence, if $z\in[[x,y]]$ and $x\le z\le y$ along the
    oriented segment, the 1-Lipschitz property gives
    \[
        u(x)-u(z)= \|x-z\|,\qquad
        u(z)-u(y)= \|z-y\|\, .
    \]

    We next prove the no-crossing condition. Let $(x,y),(x',y')\in\Gamma$,
    and suppose that $z\in[[x,y]]\cap[[x',y']]$. The contact equalities and
    the 1-Lipschitz property give the two-cycle inequality
    \[
    \begin{aligned}
        \|x-y\|+\|x'-y'\|
        &=u(x)-u(y)+u(x')-u(y')\\
        &=u(x)-u(y')+u(x')-u(y)\\
        &\le \|x-y'\|+\|x'-y\|.
    \end{aligned}
    \]
    On the other hand, the triangle inequalities through $z$ give
    \[
    \begin{aligned}
        \|x-y'\|+\|x'-y\|
        &\le \|x-z\|+\|z-y'\|+\|x'-z\|+\|z-y\|\\
        &=\|x-y\|+\|x'-y'\|.
    \end{aligned}
    \]
    Thus equality holds throughout, and in particular in both triangle
    inequalities through $z$. If $z$ is not an endpoint of both segments,
    their equality cases, together with the fact that $z$ lies on each
    segment, force the two segments to be collinear and co-oriented. If $z$
    is an endpoint of both, the only cases not already giving $x=x'$ or
    $y=y'$ are the two mixed-endpoint cases. If $y=x'$, then
    \[
        u(x)-u(y')
        =[u(x)-u(y)]+[u(x')-u(y')]
        =\|x-y\|+\|x'-y'\|,
    \]
    so the 1-Lipschitz and triangle inequalities force equality in the
    triangle inequality through $y=x'$, and the two orientations agree.
    The case $x=y'$ is analogous. Therefore, if the orientations differ, the
    segments can meet only with $x=x'$ or $y=y'$, which is
    condition~\eqref{eq:18}.

    The same equality argument proves (i) first on each transport ray. Since a
    maximal transport ray is a connected union of consistently oriented transport
    rays, and since the equalities are compatible on overlaps, the identity extends
    to all $x',y'\in S$ with $x'\le y'$:
    \[
        u(x')-u(y')=\|x'-y'\|.
    \]

    It remains to justify the two measure-theoretic regularity assertions.
    We invoke the following established regularity facts for transport rays contained
    in the contact set of a 1-Lipschitz Kantorovich potential. These facts are proved
    by Ambrosio and Pratelli in the proof of
    \cite[Theorem~6.2(ii)--(iii)]{ambrosio2003existence}; see also
    \cite[Theorem~18.9]{maggi2023optimal} for a modern treatment. For every
    $\sigma$-compact $\Gamma\subset\Gamma_u$, there exist an
    $\mathcal L^n$-negligible set $N\subset T_\Gamma$ and an increasing
    sequence of compact sets $K_h$ whose union is $T_\Gamma\setminus N$ such
    that $\tau_\Gamma|_{K_h}$ is Lipschitz for every $h$, and
    \[
        T_\Gamma^l\setminus (T_\Gamma\cup F_\Gamma)
    \]
    is $\mathcal L^n$-negligible, where
    \[
        F_\Gamma
        \coloneqq 
        \{x\in\mathbb R^n:(x,x)\in\Gamma
        \text{ and }(x,y)\notin\Gamma\text{ for every }y\ne x\}
    \]
    is the set of fixed points. In our application $\Gamma$ is
    diagonal-free, so $F_\Gamma=\emptyset$. Hence
    $T_\Gamma^l\setminus T_\Gamma$ is $\mathcal L^n$-negligible. Since
    $\mu\ll\mathcal L^n$, the compact-exhaustion statement implies condition
    (iii) in Theorem~\ref{thm:6.1'}.
\end{proof}

We will use the following simple lower-semicontinuity criterion several times.
It packages the standard portmanteau argument after shifting by a fixed
marginally integrable lower-control term.

\begin{lemma}[Lower semicontinuity under marginal lower control]
\label{lem:lsc-marginal-lower-control}
Let $\mu,\nu$ be finite positive Borel measures on $\mathbb R^n$ with equal
mass. Let $G:\mathbb R^n\to[0,\infty[$ be lower semicontinuous and assume
\[
    \int G\,d\mu+\int G\,d\nu<\infty.
\]
Let $h:\mathbb R^n\times\mathbb R^n\to\mathbb R$ be lower semicontinuous and
suppose that, for some $C\in[0,\infty[$,
\[
    h(x,y)\ge -C\bigl(1+G(x)+G(y)\bigr)
    \qquad\text{for all }x,y.
\]
Then
\[
    \gamma\mapsto \int h\,d\gamma
\]
is weakly lower semicontinuous on $\Pi(\mu,\nu)$.
\end{lemma}

\begin{proof}
Set
\[
    \widehat h(x,y)\coloneqq h(x,y)+C\bigl(1+G(x)+G(y)\bigr).
\]
Then $\widehat h$ is nonnegative and lower semicontinuous. Hence, by the
portmanteau theorem, $\gamma\mapsto\int\widehat h\,d\gamma$ is weakly lower
semicontinuous. The added term has the same finite integral for every
$\gamma\in\Pi(\mu,\nu)$, namely
\[
    C\left[\mu(\mathbb R^n)+\int G\,d\mu+\int G\,d\nu\right].
\]
Subtracting this constant gives the claim.
\end{proof}

Next, we show that the limit transport plan for concave optimizers must solve a secondary variational problem.

\begin{prop}[Limit of Concave Optimizers]
\label{prop:7.1'}
Assume {that} $\mu,\nu$ are mutually singular finite positive Borel measures with equal positive total mass and finite first moments, and $\mu\ll\mathcal{L}^n$.
Let $(\phi_\epsilon: [0, \infty[ \to [0, \infty[)_{\epsilon\in]0,1[}$ satisfy
Assumption~\ref{ass:concave-perturbation-family}, and let $\phi(\cdot)$ denote the
first-order profile defined there. Let $t_\epsilon$ be the optimal map between $\mu$ and $\nu$ for the cost
$c_\epsilon(x,y) = \phi_\epsilon(\|x-y\|)$.
Let $\gamma_\epsilon = (\mathrm{Id},t_\epsilon)_{\#}\mu$. Then any weak limit point $\gamma_0$ of the family $(\gamma_\epsilon)$ as $\epsilon \to 0^+$ is a solution to the secondary variational problem
\begin{equation*}
    (SP) \quad \min \left\{ \int_{\mathbb{R}^n \times \mathbb{R}^n} \phi(\|x-y\|) \, d\gamma : \gamma \in \Pi_1(\mu, \nu) \right\}\, ,
\end{equation*}
and the minimum is finite.
\end{prop}

\begin{proof}
Let $d(x,y) \coloneqq \|x-y\|$. For $\gamma \in \Pi(\mu, \nu)$, define the functionals
\begin{align*}
    F_\epsilon(\gamma) &\coloneqq  \int_{\mathbb{R}^n \times \mathbb{R}^n} c_\epsilon(x,y) \, d\gamma \\
    F(\gamma) &\coloneqq  \int_{\mathbb{R}^n \times \mathbb{R}^n}\|x-y\| \, d\gamma
\end{align*}
Since $\mu$ and $\nu$ have finite first moments, $\int \|x\| d\mu < \infty$ and $\int \|y\| d\nu < \infty$, which implies that for any $\gamma \in \Pi(\mu, \nu)$, $F(\gamma) = \int \|x-y\| d\gamma < \infty$ and $F_\epsilon(\gamma) < \infty$ for any $\epsilon \in ]0,1[$ using the assumption that $\phi_\epsilon(d) \leq 2(1+d)$. %
This ensures the problems are well-defined. %

First, we show that $F_\epsilon$ $\Gamma$-converges to $F$. Observe that the function $c_\epsilon(x,y) \to \|x-y\|$ pointwise as $\epsilon \to 0$, since we assumed that $\lim_{\epsilon \to 0^+} \phi_\epsilon(d) = d$.

{To show that $F_\epsilon$ $\Gamma$-converges to $F$, we need to show:
\begin{enumerate}
    \item For any $\gamma_0 \in \Pi(\mu, \nu)$ and any sequence $\gamma_\epsilon \to \gamma_0$ weakly as $\epsilon \to 0^+$  we have that $\lim \inf_{\epsilon\to 0^+} F_\epsilon(\gamma_\epsilon) \ge F(\gamma_0)$.
    \item For any $\gamma_0$, there exists a sequence $\gamma_\epsilon \to \gamma_0$ such that  $\limsup_{\epsilon\to 0^+}  F_\epsilon(\gamma_\epsilon) \le F(\gamma_0)$.
\end{enumerate}
}

We first show the liminf inequality. Let $\gamma_\epsilon \to \gamma_0$ weakly as $\epsilon \to 0^+$. We must show that $\liminf_{\epsilon\to 0^+} F_\epsilon(\gamma_\epsilon) \ge F(\gamma_0)$. The functional $F(\gamma) = \int \|x-y\| d\gamma$ is weakly lower semi-continuous, as the cost function $c(x,y) = \|x-y\|$ is non-negative and continuous. By assumption, the function $\phi$ has lower bound $\phi(d) \ge -C (1+d)$ for all $d \ge 0$. This provides a lower bound on the perturbed cost using that $\phi_\epsilon(d)$ is convex in $\epsilon$:
\[
c_\epsilon(x,y) \geq  \|x-y\| + \epsilon\phi(\|x-y\|) \ge \|x-y\| - \epsilon C(1+\|x-y\|)\, .
\]
Integrating this inequality with respect to $\gamma_\epsilon$ gives:
\begin{align}
F_\epsilon(\gamma_\epsilon)
= \int c_\epsilon(x,y) \,d\gamma_\epsilon \ge \int \|x-y\| \,d\gamma_\epsilon - \epsilon C\int (1+\|x-y\|)  d\gamma_\epsilon \, .
\label{ineq:F_epsilon_lower_bound}
\end{align}
Now,
\begin{align}
    \int (1+\|x-y\|) \,d\gamma_\epsilon \le \mu(\mathbb{R}^n) + \int \|x\| d\mu + \int \|y\| d\nu \leq B \, ,
\label{eq:1plusd-int-upper-bound}
\end{align}
for some constant $B<\infty$ since we know that $\mu$ and $\nu$ have finite first moments.

Taking the limit inferior of both sides of \eqref{ineq:F_epsilon_lower_bound} as $\epsilon \to 0^+$ and applying the weak lower semi-continuity of $F$, we get:
\[
\liminf_{\epsilon\to 0^+} F_\epsilon(\gamma_\epsilon) \ge \liminf_{\epsilon\to 0^+} \left( \int \|x-y\| \,d\gamma_\epsilon \right) - \lim_{\epsilon\to 0^+}\epsilon CB \ge F(\gamma_0) - 0.
\]
This establishes the required liminf inequality.

For the second part, the limsup inequality, choose the constant sequence $\gamma_\epsilon = \gamma_0$. Then $F_\epsilon(\gamma_0) \to F(\gamma_0)$ by the dominated convergence theorem, since $\phi_\epsilon(d) \in [0, 2(1+d)]$ and $\int (1+d) d\gamma_0 < \infty$ by the finite first moments assumption. %
{ Therefore, $F_\epsilon$ $\Gamma$-converges to $F$.}

\medskip
The fixed-marginal class $\Pi(\mu,\nu)$ is weakly compact: it is tight
because its marginals are fixed and it is weakly closed because the marginal
maps are continuous. Since $F$ is lower semicontinuous, it attains its
minimum; write
\[
    m\coloneqq \min_{\gamma\in\Pi(\mu,\nu)}F(\gamma).
\]
Next, we want to show that the rescaled functionals $F'_\epsilon(\gamma) \coloneqq  (F_\epsilon(\gamma) - m)/\epsilon$ $\Gamma$-converge to
\[ F'(\gamma) \coloneqq 
   \begin{cases}
        \int \phi(\|x-y\|) \, d\gamma & \text{if } \gamma \in \Pi_1(\mu, \nu) \\
        +\infty & \text{otherwise.}
   \end{cases}
\]
For all $\gamma \in \Pi_1(\mu, \nu)$, we know that $\int \phi(\|x-y\|) \, d\gamma$ is finite due to finite first moments of our measures, and $\phi(d) \in [-C(1+d), 4(d+1)]$. Here, $\phi(d) \leq 4(d+1)$ holds due to the following argument:  Given that the map \(\epsilon \mapsto \phi_\epsilon(d)\) is convex and \(\phi_0(d) = d\), the function $\epsilon \mapsto (\phi_\epsilon(d) - d)/\epsilon$ is non-decreasing. This property, combined with the sublinear growth assumption \(\phi_\epsilon(d) \le 2(1+d)\), provides the desired upper bound
$\phi(d) \le \frac{\phi_{1/2}(d) - d}{1/2} \le 2 \big[2(1+d) - d\big] \leq 4(d+1)$.

{We now proceed to show 1 and 2 for $F'_\epsilon$ and $F'$.}

\textbf{Liminf inequality.} Let $\gamma_\epsilon \to \gamma_0$ weakly. We must show $\liminf_{\epsilon \to 0^+} F'_\epsilon(\gamma_\epsilon) \ge F'(\gamma_0)$.
The lower bound $\phi(d)\ge-C(1+d)$ and
\eqref{eq:1plusd-int-upper-bound} give, for every
$\eta\in\Pi(\mu,\nu)$,
\[
    F'_\epsilon(\eta)
    \ge
    \frac{F(\eta)-m}{\epsilon}
    +\int\phi(d)\,d\eta
    \ge -CB.
\]
Thus the liminf cannot be $-\infty$. If it is $+\infty$ there is nothing to
prove. Otherwise choose a sequence $\epsilon_j\to0^+$ realizing this
finite liminf. Along a further tail,
$F'_{\epsilon_j}(\gamma_{\epsilon_j})$ is bounded above, and hence
\[
    F_{\epsilon_j}(\gamma_{\epsilon_j})
    =
    m+\epsilon_jF'_{\epsilon_j}(\gamma_{\epsilon_j})
    \longrightarrow m.
\]
The first Gamma-liminf inequality now yields
$F(\gamma_0)\le m$. The definition of $m$ gives the reverse inequality, so
$F(\gamma_0)=m$, that is, $\gamma_0\in\Pi_1(\mu,\nu)$.

The convexity of the map $\epsilon \mapsto \phi_\epsilon(d)$ gives the inequality $\phi_\epsilon (d) \ge d + \epsilon \phi(d)$.
Integrating over $\gamma_\epsilon$ and subtracting $m$ gives
\[
    F'_\epsilon(\gamma_\epsilon)
    \ge
    \int \phi(\|x-y\|)\,d\gamma_\epsilon
    +
    \frac{F(\gamma_\epsilon)-m}{\epsilon}.
\]
Since $F(\gamma_\epsilon)\ge m$, it follows that
\[
    F'_\epsilon(\gamma_\epsilon)
    \ge
    \int \phi(\|x-y\|)\,d\gamma_\epsilon .
\]

Since $\phi$ is finite-valued and concave on $[0,\infty[$, and $(x,y)\mapsto \|x-y\|$ is continuous, the function $h(x,y)=\phi(\|x-y\|)$ is lower semicontinuous. By Lemma~\ref{lem:lsc-marginal-lower-control}, applied with
$G(z)=\|z\|$ and this $h(x,y)$, the functional
\[
    \gamma\mapsto \int \phi(\|x-y\|)\,d\gamma
\]
is weakly lower semicontinuous on $\Pi(\mu,\nu)$.
Hence
\[
    \int \phi(\|x-y\|)\,d\gamma_0
    \le
    \liminf_{\epsilon\to0^+}
    \int \phi(\|x-y\|)\,d\gamma_\epsilon \leq \liminf_{\epsilon\to0^+} F'_\epsilon(\gamma_\epsilon)\, .
\]
Since we have already shown $\gamma_0\in\Pi_1(\mu,\nu)$, the left-hand side is
$F'(\gamma_0)$. Therefore
\[
    \liminf_{\epsilon\to0^+}F'_\epsilon(\gamma_\epsilon)
    \ge
    F'(\gamma_0)\, .
\]

\textbf{Limsup inequality.} Without loss of generality, assume that
$\gamma_0\in\Pi_1(\mu,\nu)$. We choose the constant recovery sequence
$\gamma_\epsilon=\gamma_0$. Since $F(\gamma_0)=m$,
\[
    F'_\epsilon(\gamma_0)
    =
    \int \frac{\phi_\epsilon(\|x-y\|)-\|x-y\|}{\epsilon}\,d\gamma_0.
\]
By definition of the first-order profile, the integrand converges pointwise to
$\phi(\|x-y\|)$. We now show that this integrand is dominated by a function that is integrable with respect to $\gamma_0$.

Fix $\epsilon_0\in]0,1[$. By convexity of
$\epsilon\mapsto \phi_\epsilon(d)$ and $\phi_\epsilon(d)\to d$, the secant
quotients
$
   q_\epsilon(d) \coloneq \frac{\phi_\epsilon(d)-d}{\epsilon}
$
are nondecreasing in $\epsilon$. Therefore, for $0<\epsilon\le\epsilon_0$, we have
\[
    \phi(d)\le q_\epsilon(d)\le q_{\epsilon_0}(d).
\]
But by assumption,
$
    \phi(d)\ge -C(1+d),
$ and
$
    q_{\epsilon_0}(d)
    =
    \frac{\phi_{\epsilon_0}(d)-d}{\epsilon_0}
    \le
    \frac{2(1+d)}{\epsilon_0}
$.
Thus $|q_\epsilon(d)|\le \max\{C, 2/\epsilon_0\} (1+d)$ for $0<\epsilon\le\epsilon_0$.
Since $d(x,y)\le \|x\|+\|y\|$ and $\gamma_0$ has fixed marginals with finite
first moments, this dominating function is $\gamma_0$-integrable.

Dominated convergence then yields
\[
    \lim_{\epsilon\to0^+}F'_\epsilon(\gamma_0)
    =
    \int \phi(d(x,y))\,d\gamma_0
    =
    F'(\gamma_0) \, .
\]

\smallskip
For completeness, the hypotheses of the fundamental theorem of
Gamma-convergence are automatic here. All functionals are considered on the
compact space $\Pi(\mu,\nu)$, so the family is equicoercive. Moreover,
\[
    \operatorname*{argmin}_{\Pi(\mu,\nu)}F'_\epsilon
    =
    \operatorname*{argmin}_{\Pi(\mu,\nu)}F_\epsilon,
\]
because $F'_\epsilon=(F_\epsilon-m)/\epsilon$ with $\epsilon>0$. Therefore
\cite[Theorem~4.1]{ambrosio2003existence} implies that every weak limit point
of the minimizing plans $\gamma_\epsilon$ minimizes $F'$. The recovery
sequence above also shows that the minimum of $F'$ is finite.
\end{proof}

Finally, we show that the canonical fixed-support construction furnished
by Theorem~\ref{thm:6.1'} is the same for every admissible strictly concave
secondary profile, thereby defining an intrinsic generalized EC map.

\begin{theorem}[Unique minimizer of the secondary problem]
\label{thm:7.2'}
Let $\mu,\nu$ be mutually singular finite positive Borel measures on $\mathbb{R}^n$ with equal positive total mass and finite first moments, and assume $\mu\ll\mathcal{L}^n$.
Consider the secondary variational problem
\begin{equation*}
    (SP) \quad \min \left\{ \int_{\mathbb{R}^n \times \mathbb{R}^n} \psi(\|x-y\|) \, d\gamma : \gamma \in \Pi_1(\mu, \nu) \right\},
\end{equation*}
where $\psi:[0,\infty[\to\mathbb R$ is strictly concave and finite-valued and satisfies
$\psi(d)\ge -K_\psi(1+d)$ for all $d\ge0$ and some
$K_\psi\in[0,\infty[$.
The minimum value is finite, and (SP) has a unique minimizer
$\gamma_\#$ induced by a transport map $t_\#$. On almost every maximal
transport ray, this map is the one-dimensional excursion-coupling map.
Moreover, $\gamma_\#$, and hence $t_\#$ up to $\mu$-a.e. equality, is the same
for every choice of $\psi$ satisfying these assumptions. We call $t_\#$ the
generalized EC map associated with $(\mu,\nu)$.
\end{theorem}
\begin{proof}

Existence follows by the direct method. Indeed, $\Pi_1(\mu,\nu)$ is weakly
compact. Since $\psi$ is finite-valued and concave on $[0,\infty[$, the
function
\[
    h(x,y)\coloneqq \psi(\|x-y\|)
\]
is lower semicontinuous. Moreover,
\[
    h(x,y)\ge -K_\psi(1+\|x-y\|)
    \ge -K_\psi(1+\|x\|+\|y\|).
\]
Thus Lemma~\ref{lem:lsc-marginal-lower-control}, applied with
$G(z)=\|z\|$ and this $h$, shows that the secondary functional is weakly
lower semicontinuous on $\Pi(\mu,\nu)$, hence on $\Pi_1(\mu,\nu)$.
Also, since $\psi$ is finite-valued and concave on $[0,\infty[$, it admits an
affine upper bound in $d$. Together with
$\psi(d)\ge -K_\psi(1+d)$ and the finite first moments of $\mu,\nu$, this shows
that $\int \psi(\|x-y\|)\,d\gamma$ is finite for every
$\gamma\in\Pi_1(\mu,\nu)$.
The secondary functional therefore attains its minimum on $\Pi_1(\mu,\nu)$.

Fix a 1-Lipschitz Kantorovich potential $u$ given by
Lemma~\ref{lem:linear-cost-contact-set}, and let $\Gamma_u$ be its contact set.
Writing
\[
    \Delta\coloneqq \{(x,x):x\in\mathbb R^n\},
    \qquad
    \Gamma^*\coloneqq \Gamma_u\setminus\Delta,
\]
we claim that $\Gamma^*$ is $\sigma$-compact. Indeed, since $u$ is continuous,
$\Gamma_u$ is closed. For $m,k\ge1$, let
\[
    K_{m,k}
    \coloneqq 
    \Gamma_u
    \cap\{(x,y):\|(x,y)\|\le m\}
    \cap\{(x,y):\|x-y\|\ge 1/k\}.
\]
Each $K_{m,k}$ is compact, and
\[
    \Gamma^*=\bigcup_{m,k\ge1}K_{m,k}.
\]
Thus $\Gamma^*$ is a fixed $\sigma$-compact diagonal-free subset of
$\Gamma_u$. By Theorem~\ref{thm:6.2'}, it satisfies the no-crossing condition
and conditions (ii)--(iii) of Theorem~\ref{thm:6.1'}; condition (i) follows
from the assumption $\mu\ll\mathcal L^n$.

Let $\gamma$ be any minimizer of (SP). Since
$\gamma\in\Pi_1(\mu,\nu)$, Lemma~\ref{lem:linear-cost-contact-set} gives
\[
    \gamma(\Gamma_u)=\mu(\mathbb R^n).
\]
Moreover, mutual singularity implies that $\gamma$ gives no mass to the
diagonal. Indeed, choose a Borel set $A$ such that
$\mu(A^c)=0$ and $\nu(A)=0$. Then
\[
    \gamma(\Delta)
    \le
    \gamma(A^c\times\mathbb R^n)
    +
    \gamma(\mathbb R^n\times A)
    =
    \mu(A^c)+\nu(A)
    =
    0.
\]
Consequently, $\gamma$ is concentrated on $\Gamma^*$.

Applying Theorem~\ref{thm:6.1'} to $\gamma$, with the fixed supporting set
$\Gamma^*$ and with $\phi=\psi$, produces the plan
$\gamma_\#^{\Gamma^*}$, canonical for this fixed $\Gamma^*$. Since
$\Gamma^*\subset\Gamma_u$ and $\Gamma_u$ is $c$-cyclically monotone by
Lemma~\ref{lem:linear-cost-contact-set}, $\Gamma^*$ is also
$c$-cyclically monotone. Conclusion (a) of Theorem~\ref{thm:6.1'} therefore
gives
\[
    \gamma_\#^{\Gamma^*}\in\Pi_1(\mu,\nu).
\]
Conclusion (c) gives
\[
    \int \psi(\|x-y\|)\,d\gamma_\#^{\Gamma^*}
    \le
    \int \psi(\|x-y\|)\,d\gamma.
\]
Because $\gamma$ minimizes (SP) and $\gamma_\#^{\Gamma^*}$ is admissible for
(SP), the reverse inequality also holds. Hence equality holds, and the
strictness clause in Theorem~\ref{thm:6.1'} yields
\[
    \gamma=\gamma_\#^{\Gamma^*}.
\]
Since $\gamma$ was an arbitrary minimizer, (SP) has the unique minimizer
$\gamma_\#^{\Gamma^*}$.

It remains to show that this plan does not depend on the auxiliary choices.
Let $\widetilde u$ be another Kantorovich potential furnished by
Lemma~\ref{lem:linear-cost-contact-set}, and set
\[
    \widetilde\Gamma^*
    \coloneqq 
    \Gamma_{\widetilde u}\setminus\Delta.
\]
Repeating the preceding argument with $\widetilde u$ shows that
$\gamma_\#^{\widetilde\Gamma^*}$ is also the unique minimizer of the same
problem (SP). Consequently,
\[
    \gamma_\#^{\widetilde\Gamma^*}
    =
    \gamma_\#^{\Gamma^*}.
\]
Thus the resulting plan is independent of the chosen Kantorovich potential
and of the corresponding supporting set.

Finally, let $\widetilde\psi$ be any other strictly concave function satisfying
the assumptions of the theorem. Keeping $u$ and $\Gamma^*$ fixed, repeat the
preceding argument with $\widetilde\psi$ in place of $\psi$. The plan furnished
by Theorem~\ref{thm:6.1'} is still $\gamma_\#^{\Gamma^*}$, since that plan
depends only on $\mu,\nu,\Gamma^*$ and not on the function used in conclusion
(c). Hence $\gamma_\#^{\Gamma^*}$ is also the unique minimizer of the secondary
problem with objective $\widetilde\psi$.

We may therefore denote this common plan by $\gamma_\#$ and its inducing map,
defined up to $\mu$-a.e. equality, by $t_\#$. This is the generalized EC map
associated with $(\mu,\nu)$.
\end{proof}

We are now ready to provide a proof of the main theorem.
\begin{proof}[Proof of Theorem~\ref{thm:main-theorem}]
Let $M\coloneqq \mu(\mathbb R^n)=\nu(\mathbb R^n)>0$ and normalize
$\bar\mu\coloneqq M^{-1}\mu$, $\bar\nu\coloneqq M^{-1}\nu$. Scaling by $M^{-1}$ is a
bijection between the finite coupling class $\Pi(\mu,\nu)$ and the probability
coupling class $\Pi(\bar\mu,\bar\nu)$, and it multiplies every transport
objective by $M^{-1}$. The hypotheses of
\cite[Theorem~3.3]{pegon2013concave-general} hold because each
$\phi_\epsilon$ is finite-valued, nonnegative, increasing, and strictly
concave. That theorem gives a unique probability optimizer induced by a map;
scaling it by $M$ gives the unique optimizer in $\Pi(\mu,\nu)$, induced by the
same map $t_\epsilon$.

We now prove convergence of the plans, and then of the maps. Let
$\gamma_\epsilon=(\mathrm{Id},t_\epsilon)_\#\mu$. Since $\Pi(\mu,\nu)$ is
weakly compact, every sequence $\epsilon_j\to 0^+$ has a subsequence such
that $\gamma_{\epsilon_j}$ converges weakly to some $\gamma_0\in\Pi(\mu,\nu)$.
By Proposition~\ref{prop:7.1'}, every such subsequential limit solves the
secondary problem $(SP)$ with secondary objective given by the first-order
profile $\phi$. By Theorem~\ref{thm:7.2'}, applied with $\psi=\phi$, this secondary
problem has the unique minimizer $\gamma_\#$ induced by the intrinsic
generalized EC map $t_\#$; the same theorem shows that this plan and map do
not depend on $\phi$. Hence every subsequential weak limit of
$(\gamma_\epsilon)$ is $\gamma_\#$. The narrow topology on finite measures of fixed mass over the Polish
space $\mathbb R^n\times\mathbb R^n$ is metrizable. If the full family did not
converge to $\gamma_\#$, a sequence $\epsilon_j\to0^+$ would remain
outside some neighborhood of $\gamma_\#$; compactness would give a convergent
subsequence, whose limit would have to be $\gamma_\#$, a contradiction.
Therefore $\gamma_\epsilon\to\gamma_\#$ weakly as $\epsilon\to0^+$.

It remains to pass from weak convergence of graph plans to convergence in
$\mu$-measure of the maps. Fix $\eta>0$. By Lusin's theorem, for every
$\delta>0$ there is a compact set $K\subset\mathbb R^n$ such that
$\mu(K^c)<\delta$ and $t_\#|_K$ is continuous. Then
\[
    A_{\eta,K}\coloneqq \{(x,y):x\in K,\ \|y-t_\#(x)\|\ge \eta\}
\]
 is closed, and $\gamma_\#(A_{\eta,K})=0$ because
$\gamma_\#=(\mathrm{Id},t_\#)_\#\mu$. Therefore, by the portmanteau theorem,
\[
    \limsup_{\epsilon\to0^+}\gamma_\epsilon(A_{\eta,K})=0.
\]
Since $\gamma_\epsilon=(\mathrm{Id},t_\epsilon)_\#\mu$,
\[
    \mu\{x:\|t_\epsilon(x)-t_\#(x)\|\ge\eta\}
    \le \mu(K^c)+\gamma_\epsilon(A_{\eta,K}).
\]
Taking $\limsup_{\epsilon\to0^+}$ and then letting $\delta\to 0^+$ gives
$t_\epsilon\to t_\#$ in $\mu$-measure.

\end{proof}

\section{Proof of Theorem~\ref{thm:6.1'}}
\label{sec:proof-of-6.1'}

This section proves Theorem~\ref{thm:6.1'} by disintegrating the given plan along oriented maximal transport rays, replacing each conditional plan by its excursion coupling, and gluing the conditional maps.  Ambrosio and Pratelli's oriented-ray disintegration makes each conditional plan forward along its ray \cite{ambrosio2003existence}.
Section~\ref{subsec:1D-theory} first develops the required one-dimensional result for mutually singular, stochastically ordered finite measures with atomless source; Section~\ref{subsec:proof-of-thm-6.1'} then carries out the measurable raywise construction.

\subsection{1D Theory: The Excursion Coupling}
\label{subsec:1D-theory}

Juillet \cite{juillet2020solution} studies optimal transport on $\mathbb{R}$ for the concave power costs
$c_\varepsilon(x,y)=|x-y|^{1-\varepsilon}$, $\varepsilon\in]0,1[$, and characterizes the limiting optimal plan as $\varepsilon\to0^+$. He shows that the limit plan is the so-called \emph{excursion coupling}; and that, for each $\varepsilon\in]0,1[$, the excursion coupling is the unique minimizer of $\int |x-y|^{1-\varepsilon}\,d\gamma$ among the $L^1$-optimal couplings. We generalize this result to the broader class of concave costs we consider, for measures satisfying the conditions stated in Theorem~\ref{thm:6.1'}. In particular, we are concerned only with the special case of the excursion coupling for a transport problem on the real line satisfying the following conditions:  (i) the two measures $\mu$ and $\nu$ are mutually singular with equal mass and finite first moments, (ii) the source measure $\mu$ is atomless, and (iii) the measures are stochastically ordered, in particular $F_\nu(x) \le F_\mu(x)$ for all $x \in \mathbb{R}$.
Under these conditions, the excursion coupling corresponds to a deterministic transport map which we now describe.
The first two conditions ensure that the coupling reduces to a deterministic map, while the third condition simplifies the description of the map.

We will reuse from \cite{juillet2020solution} the following definition of the completed graph of a function.
For a  function $H:\mathbb{R}\to\mathbb{R}$ which is right continuous with left limits (i.e., cadlag) and of bounded variation, write
$H(x^-)\coloneqq\lim_{z\to x^-}H(z)$. Its \emph{completed graph} is
\[
    {\operatorname{CG}(H)\coloneqq
    \{(x,h)\in\mathbb{R}^2:
    \min\{H(x^-),H(x)\}\le h\le \max\{H(x^-),H(x)\}\}.}
\]
{Thus at a continuity point $x$ this contains only the usual graph point
$(x,H(x))$, while at a jump it contains the vertical segment joining
$H(x^-)$ to $H(x)$. In our application $H=F_\rho\coloneqq F_\mu-F_\nu$, where
$\rho=\mu-\nu$ and $F_\rho(x)=\rho(]-\infty,x])$; the vertical segment at
$x$ records the atom $\rho(\{x\})=F_\rho(x)-F_\rho(x^-)$.}

We use the following mutually singular, ordered-case form of Juillet's \emph{excursion coupling}. Let $\mu,\nu$ be mutually singular finite positive Borel measures on $\mathbb{R}$ with equal positive mass $M$, and assume that $\nu$ stochastically dominates $\mu$, i.e., $F_\rho(x)\coloneqq F_\mu(x)-F_\nu(x)\ge 0$ for all $x\in\mathbb{R}$.
Juillet's construction is made on the completed graph
$\operatorname{CG}(F_\rho)$ rather than on the ordinary graph of
$F_\rho$. By \cite[Proposition~3.2]{juillet2020solution}, for Lebesgue-a.e.
level $h>0$, the horizontal section
\[
    \operatorname{CG}(F_\rho)\cap(\mathbb{R}\times\{h\})
\]
is finite and consists only of genuine crossings of the level $h$: locally the
completed graph passes from below to above the line $y=h$, or from above to
below it. Since $F_\rho(-\infty)=F_\rho(+\infty)=0<h$, these crossings start
with an entrance into the region above level $h$, then alternate between exits
and entrances, and end with an exit. Thus their $x$-coordinates can be written
as
\[
    x_1^h<x_2^h<\cdots<x_{2N(h)}^h,
\]
where $x_{2i-1}^h$ is an entrance and $x_{2i}^h$ is the next exit. Juillet's ordered-case construction pairs each entrance with the next exit. Equivalently, for every bounded Borel test function $\psi$,
\[
    \int_{\mathbb{R}^2}\psi(x,y)\,d\gamma_{\mathrm{EC}}(x,y)
    \coloneqq 
    \int_0^\infty \sum_{i=1}^{N(h)}
    \psi\!\left(x_{2i-1}^h,x_{2i}^h\right)\,dh,
\]
where the exceptional null set of levels is immaterial.
For probability measures this is exactly Juillet's construction in
\cite[Theorem~1.1]{juillet2020solution}; for common mass $M$, apply that
construction to $M^{-1}\mu$ and $M^{-1}\nu$ and multiply the resulting coupling
by $M$. Hence the measure defined above has marginals $\mu$ and $\nu$.
Since $x_{2i-1}^h<x_{2i}^h$ for every paired crossing, all EC
pairs are forward.

The following theorem states that the EC coupling in the mutually singular, ordered case with atomless $\mu$ is Monge and forward.

\begin{theorem}[Excursion Coupling Map for mutually singular, ordered measures, and atomless $\mu$]
\label{thm:EC-map-gives-a-plan-1D}
{Let $\mu,\nu$ be mutually singular finite positive Borel measures on $\mathbb{R}$ with equal positive mass and finite first moments. Assume that $\mu$ is atomless and that $\nu$ stochastically dominates $\mu$, i.e., $F_\nu(x)\le F_\mu(x)$ for all $x\in\mathbb{R}$. Let $\gamma_{\mathrm{EC}}$ be Juillet's excursion coupling. Then, there exists a Borel map $T_{\mathrm{EC}}$ such that}
\[
    \gamma_{\mathrm{EC}}=(\mathrm{Id},T_{\mathrm{EC}})_\#\mu
\]
and $\gamma_{\mathrm{EC}}\in\Pi(\mu,\nu)$ is forward:
\[
    \gamma_{\mathrm{EC}}(\{(x,y)\in\mathbb{R}^2:y<x\})=0\, .
\]
\end{theorem}

\begin{proof}

Let $M\coloneqq \mu(\mathbb R)=\nu(\mathbb R)$ and normalize all three
measures before applying Juillet's probability theorem:
\[
    (\bar\mu,\bar\nu,\bar\gamma_{\mathrm{EC}})
    \coloneqq 
    (M^{-1}\mu,M^{-1}\nu,M^{-1}\gamma_{\mathrm{EC}}).
\]
The completed-graph construction is homogeneous under this
normalization.  Thus \cite[Theorem~1.1]{juillet2020solution}, applied to the
probability measures $\bar\mu,\bar\nu$, gives
$\bar\gamma_{\mathrm{EC}}\in\Pi(\bar\mu,\bar\nu)$; multiplying by $M$ gives
$\gamma_{\mathrm{EC}}\in\Pi(\mu,\nu)$. Since
$F_{\bar\mu}-F_{\bar\nu}\ge0$, every pairing at a level $h>0$ is forward.
Although the completed graph may meet the level $h=0$, the EC coupling is
obtained by integrating the levelwise pairings against Lebesgue measure $dh$.
Since $\mathcal L^1(\{0\})=0$, restricting to levels $h>0$ does not change
$\gamma_{\mathrm{EC}}$.

It remains to justify the Borel map assertion.
In the notation of the completed-graph construction above, the proof of
\cite[Proposition~3.6]{juillet2020solution} shows that, away from atoms of
the source, the paired crossings
$(x_{2i-1}^h,x_{2i}^h)$ contain at most one pair with any given first
coordinate. Denote Juillet's paired-crossing relation by
$\Gamma_{\mathrm J}$. Since $\bar\gamma_{\mathrm{EC}}$ is concentrated on
$\Gamma_{\mathrm J}$, choose a Borel set
$R\subset\Gamma_{\mathrm J}$ with $\bar\gamma_{\mathrm{EC}}(R)=1$, and set
$R_x\coloneqq\{y:(x,y)\in R\}$.
By the disintegration theorem on Polish spaces,
there is a Borel probability kernel
$x\mapsto\kappa_x\in\mathcal P(\mathbb R)$ such that
\[
    \bar\gamma_{\mathrm{EC}}(dx,dy)
    =\kappa_x(dy)\,\bar\mu(dx).
\]
Disintegrating $\bar\gamma_{\mathrm{EC}}(R)=1$ gives
\[
    1=\int_{\mathbb R}\kappa_x(R_x)\,d\bar\mu(x),
\]
so $\kappa_x(R_x)=1$ for $\bar\mu$-a.e. $x$. Since
$\bar\mu\perp\bar\nu$ and $\bar\mu$ is atomless, the preceding
singleton-fiber conclusion implies that $\kappa_x$ is a Dirac mass for
$\bar\mu$-a.e. $x$.
The set of Dirac masses is closed in $\mathcal P(\mathbb R)$, and
$y\mapsto\delta_y$ is a Borel isomorphism onto this set. Therefore
$T_{\mathrm{EC}}(x)\coloneqq \delta^{-1}(\kappa_x)$ defines a Borel map on a
Borel set of full $\bar\mu$-measure. Assigning the fixed value $0$ on its
Borel complement gives a Borel map on $\mathbb R$ satisfying
$\kappa_x=\delta_{T_{\mathrm{EC}}(x)}$ for $\bar\mu$-a.e. $x$.

Scaling back does not change the
conditional kernels, and hence
$\gamma_{\mathrm{EC}}=(\mathrm{Id},T_{\mathrm{EC}})_\#\mu$.
\end{proof}

\begin{lemma}[Forward plans are the ordered $L^1$ optimizers]
\label{lem:forward-equals-l1-ordered}
Let $\mu,\nu$ be finite positive Borel measures on $\mathbb{R}$ with equal mass and finite first moments. Assume that $\nu$ stochastically dominates $\mu$, i.e. $F_\nu(t)\le F_\mu(t)$ for all $t\in\mathbb{R}$. Let
\[
    \mathcal A\coloneqq \{\gamma\in\Pi(\mu,\nu):
    \gamma(\{(x,y):x\le y\})=\mu(\mathbb{R})\}\, .
\]
Then $\mathcal A$ is nonempty and $\mathcal A=\Pi_1(\mu,\nu)$.
\end{lemma}

\begin{proof}
Nonemptiness follows from the quantile coupling. If
$G_\alpha(s)\coloneqq \inf\{x:F_\alpha(x)\ge s\}$ for $s\in]0,M[$, where
$M\coloneqq \mu(\mathbb{R})=\nu(\mathbb{R})$, then $F_\nu\le F_\mu$ implies
$G_\mu(s)\le G_\nu(s)$, and
$(G_\mu,G_\nu)_\#\mathcal L^1|_{]0,M[}$ belongs to $\mathcal A$.

For any coupling $\gamma$,
\[
    |y-x|=(y-x)+2(x-y)_+ .
\]
Integrating gives
\[
    \int |y-x|\,d\gamma
    =
    \int y\,d\nu-\int x\,d\mu
    +2\int (x-y)_+\,d\gamma
    \ge
    \int y\,d\nu-\int x\,d\mu .
\]
Every forward plan attains equality. Conversely, equality forces
$(x-y)_+=0$ $\gamma$-a.e., hence the plan is forward.
\end{proof}

\begin{lemma}[Lexicographic variational principle]
\label{lem:lexicographic-variational-main}
Let $\mu,\nu$ be finite positive Borel measures on $\mathbb{R}$ with equal mass and finite first moments. Let $c_2:\mathbb{R}^2\to\mathbb{R}$ be lower semicontinuous with lower bound $c_2(x,y) \geq -C(1+|x-y |)$ for all $x,y \in \mathbb{R}$ for some constant $C \in [0,\infty[$, and let $\gamma$ minimize $\int c_2\,d\eta$ among the $L^1$-optimal plans, with finite secondary value. Then there are a 1-Lipschitz function $u:\mathbb{R}\to\mathbb{R}$ and a Borel set
\[
    S\subset\Gamma_u\coloneqq \{(x,y)\in\mathbb{R}^2: |x-y|=u(x)-u(y)\}
\]
with full $\gamma$-mass such that, for every finite family
$(x_i,y_i)_{i=1}^N\subset S$ and every permutation $\tau$,
\begin{equation}
    \sum_i |x_i-y_i|\le \sum_i |x_i-y_{\tau(i)}|\, ,
    \label{eq:lex-primary-main}
\end{equation}
and, if equality holds, then
\begin{equation}
    \sum_i c_2(x_i,y_i)\le \sum_i c_2(x_i,y_{\tau(i)})\, .
    \label{eq:lex-secondary-main}
\end{equation}
\end{lemma}

\begin{proof}
Let $M\coloneqq \mu(\mathbb{R})=\nu(\mathbb{R})$. Apply
    Lemma~\ref{lem:linear-cost-contact-set}, in dimension one, to the pair
    $(\mu,\nu)$. Let $u:\mathbb{R}\to\mathbb{R}$ and
    \[
        \Gamma_u\coloneqq \{(x,y)\in\mathbb{R}^2: |x-y|=u(x)-u(y)\}
    \]
    be the resulting Kantorovich potential and contact set. Set
    \[
        \mathcal K_{\Gamma_u}\coloneqq 
        \{\eta\in\Pi(\mu,\nu):\eta(\Gamma_u)=M\}.
    \]
    Lemma~\ref{lem:linear-cost-contact-set} gives
    $\Pi_1(\mu,\nu)=\mathcal K_{\Gamma_u}$. Hence $\gamma$ is a
    $c_2$-minimizer over this restricted coupling class.

Encode the restriction to $\Gamma_u$ and the linear lower bound at the same time by setting
\[
    \widehat c_2(x,y)\coloneqq 
    \begin{cases}
        c_2(x,y)+C(1+|x|+|y|), & (x,y)\in\Gamma_u,\\
        +\infty, & (x,y)\notin\Gamma_u .
    \end{cases}
\]
Since $c_2$ is lower semicontinuous and $\Gamma_u$ is closed, $\widehat c_2$ is lower semicontinuous. Moreover, the lower bound
$c_2(x,y)\ge -C(1+|x-y|)$ gives
\[
    \widehat c_2(x,y)
    \ge C(|x|+|y|-|x-y|)
    \ge 0 .
\]
For $\eta\in\mathcal K_{\Gamma_u}$ we have
\[
    \int \widehat c_2\,d\eta
    =
    \int c_2\,d\eta
    +
    C\int (1+|x|+|y|)\,d\eta,
\]
and the last term is constant over $\Pi(\mu,\nu)$, using that the measures have finite first moments. If
$\eta\notin\mathcal K_{\Gamma_u}$, then $\int\widehat c_2\,d\eta=+\infty$.
Thus minimizing $c_2$ over $\mathcal K_{\Gamma_u}$ is equivalent to minimizing
$\widehat c_2$ over all couplings in $\Pi(\mu,\nu)$; in particular, $\gamma$
minimizes $\widehat c_2$.

The cost $\widehat c_2:\mathbb R^2\to[0,+\infty]$ is lower
semicontinuous, and the assumed finite secondary value $\int c_2\,d\gamma$ together with the
finite first moments gives
$\int\widehat c_2\,d\gamma<+\infty$.  
The case $M=0$ is
vacant.
If $M>0$, apply the necessity part of
\cite[Theorem~3.2]{ambrosio2003existence} to the probability plan
$M^{-1}\gamma$ with marginals $M^{-1}\mu,M^{-1}\nu$. Intersecting the $\widehat c_2$-cyclically monotone Borel set
supplied by that theorem with $\Gamma_u$ gives a
full-mass Borel set $S\subset\Gamma_u$ such that, whenever
$(x_i,y_i)\subset S$ and every rearranged pair $(x_i,y_{\tau(i)})$ also lies in
$\Gamma_u$,
\begin{align*}
    \sum_i \widehat c_2(x_i,y_i)
    \le
    \sum_i \widehat c_2(x_i,y_{\tau(i)}) \, .
\end{align*}
On these terms $\widehat c_2(x,y)=c_2(x,y)+C(1+|x|+|y|)$, and the added terms cancel
under the permutation. Hence
\begin{align}
    \sum_i c_2(x_i,y_i)
    \le
    \sum_i c_2(x_i,y_{\tau(i)})\, .
\label{eq:restricted-c2-main}
\end{align}

Now fix $(x_i,y_i)\subset S$ and an arbitrary permutation $\tau$. Since the source and target lists are unchanged,
\[
    \sum_i(u(x_i)-u(y_i))
    =
    \sum_i(u(x_i)-u(y_{\tau(i)})).
\]
The original pairs lie in $\Gamma_u$, while $u$ is 1-Lipschitz, hence
\[
    \sum_i |x_i-y_i|
    =
    \sum_i(u(x_i)-u(y_{\tau(i)}))
    \le
    \sum_i |x_i-y_{\tau(i)}|,
\]
which proves \eqref{eq:lex-primary-main}. If equality holds, then every nonnegative gap
\[
    |x_i-y_{\tau(i)}|-u(x_i)+u(y_{\tau(i)})
\]
is zero, so every rearranged pair belongs to $\Gamma_u$. Applying \eqref{eq:restricted-c2-main} gives \eqref{eq:lex-secondary-main}.
\end{proof}

\begin{lemma}[Secondary swapping for a secondary minimizer]
\label{lem:secondary-swap-general-profile}
Let $\mu,\nu$ be finite positive Borel measures on $\mathbb{R}$ with equal mass and finite first moments. Let $\psi:[0,\infty[\to\mathbb{R}$ be lower semicontinuous with lower bound $\psi(d) \geq -C(1+ d)$ for all $d \in [0,\infty[$ for some constant $C \in [0,\infty[$, and suppose that $\gamma$ minimizes
\[
    \int_{\mathbb{R}^2}\psi(|y-x|)\,d\eta
\]
among the $L^1$-optimal plans $\eta\in\Pi_1(\mu,\nu)$, and assume the displayed secondary value is finite. Then there is a Borel set $S\subset\mathbb{R}^2$ with $\gamma(S)=\gamma(\mathbb{R}^2)$ such that for all $(x,y),(x',y')\in S$,
\begin{equation}
    |y-x|+|y'-x'|\le |y'-x|+|y-x'|,
    \label{eq:primary-swap-main}
\end{equation}
and, whenever equality holds in \eqref{eq:primary-swap-main},
\begin{equation}
    \psi(|y-x|)+\psi(|y'-x'|)
    \le
    \psi(|y'-x|)+\psi(|y-x'|).
    \label{eq:secondary-swap-main}
\end{equation}
\end{lemma}

\begin{proof}
Set $c_2(x,y)=\psi(|y-x|)$. Then $c_2$ is lower semicontinuous with lower bound $c_2(x,y) \geq -C(1+| x-y |)$ for all $x,y \in \mathbb{R}$ for $C\in[0,\infty[$. Apply Lemma~\ref{lem:lexicographic-variational-main} to this secondary cost. The two-point permutation in that lemma is exactly the rerouting
\[
    (x,y),(x',y')\longmapsto (x,y'),(x',y).
\]
The primary inequality in the lexicographic lemma gives \eqref{eq:primary-swap-main}. If the primary costs tie, the secondary inequality in the lexicographic lemma gives \eqref{eq:secondary-swap-main}.
\end{proof}

\begin{lemma}[Strict concavity comparison]
    \label{lem:strict-concavity-arch-comparison}
    Let $\psi$ be finite-valued and strictly concave on $[0,\infty[$. If
    $a,c>0$ and $b\ge0$, then
    \[
        \psi(a+b)+\psi(b+c)>\psi(b)+\psi(a+b+c)\, .
    \]
\end{lemma}

\begin{proof}
Set $\zeta=a+b+c$. The points $a+b$ and $b+c$ lie strictly between $b$ and $\zeta$, and their sum equals $b+\zeta$. Hence for some $\lambda\in]0,1[$,
\[
    a+b=\lambda b+(1-\lambda)\zeta,\qquad
    b+c=(1-\lambda)b+\lambda \zeta.
\]
Applying strict concavity to these two identities and then adding the resulting inequalities gives the claim.
\end{proof}

For any distance-cost profile $\psi:[0,\infty[\to\mathbb R$ for which the integral is well-defined, write
\[
   J_\psi(\gamma)\coloneqq \int_{\mathbb{R}^2}\psi(|y-x|)\,d\gamma(x,y)\, .
\]

\begin{theorem}[Variational Property of the Excursion Coupling]
    \label{thm:5.2'}
    Let $\mu,\nu$ be mutually singular finite positive Borel measures on $\mathbb{R}$
    with equal positive mass and finite first moments. Assume that $\mu$ is atomless
    and that $\nu$ stochastically dominates $\mu$. Let $\gamma_{\mathrm{EC}}$ be the
    excursion coupling for this pair. %
    Let
    \[
        \mathcal A\coloneqq \{\gamma\in\Pi(\mu,\nu):
        \gamma(\{(x,y):x\le y\})=\mu(\mathbb{R})\}.
    \]
    Then $\mathcal A$ is nonempty and $\mathcal A=\Pi_1(\mu,\nu)$. Moreover, if
$\phi:[0,\infty[\to\mathbb{R}$ is finite-valued and concave, and satisfies
\[
    \phi(d)\ge -C(1+d)\qquad\text{for all }d\ge0
\]
for some $C\in[0,\infty[$, then
    \[
        J_\phi(\gamma_{\mathrm{EC}})
        =
        \min_{\gamma\in\mathcal A} J_\phi(\gamma)\, .
    \]
    If $\phi$ is strictly concave, then $\gamma_{\mathrm{EC}}$ is the unique
    minimizer.
    \end{theorem}

\begin{proof}
Theorem~\ref{thm:EC-map-gives-a-plan-1D} gives
$\gamma_{\mathrm{EC}}\in\Pi(\mu,\nu)$ and
\[
    \gamma_{\mathrm{EC}}(\{(x,y):y<x\})=0.
\]
Equivalently,
\[
    \gamma_{\mathrm{EC}}(\{(x,y):x\le y\})=\mu(\mathbb{R}),
\]
so $\gamma_{\mathrm{EC}}\in\mathcal A$.
Lemma~\ref{lem:forward-equals-l1-ordered} gives that $\mathcal A=\Pi_1(\mu,\nu)$.

First assume that $\phi$ is strictly concave.

Since $\phi$ is finite-valued and concave on $[0,\infty[$, it is lower
semicontinuous. The linear lower bound gives
\[
    \phi(|y-x|)\ge -C(1+|y-x|)
    \ge -C(1+|x|+|y|).
\]
Thus
\[
    (x,y)\mapsto \phi(|y-x|)+C(1+|x|+|y|)
\]
is nonnegative and lower semicontinuous. By the portmanteau theorem, its integral is
lower semicontinuous under weak convergence of couplings. Since the added
term has the same finite integral for every coupling in $\Pi(\mu,\nu)$, it
follows that $J_\phi$ is lower semicontinuous on $\Pi(\mu,\nu)$, hence on
$\mathcal A$. Also, finite-valued concavity gives an affine upper bound on $\phi$, while the
assumed lower bound gives affine lower control; hence $J_\phi$ is finite on
$\mathcal A$ because all couplings have finite first moments.

The secondary problem on $\mathcal A$ has a minimizer by the direct method.
Indeed, $\Pi(\mu,\nu)$ is compact for weak convergence: fixed marginals imply
tightness, and the marginal constraints are closed under weak convergence.
Moreover, $\mathcal A$ is weakly closed in $\Pi(\mu,\nu)$: if $\gamma_n\in\mathcal A$ and $\gamma_n\to\gamma$ weakly, then, since $\{(x,y):x\le y\}$ is closed, the portmanteau theorem gives
\[
    \mu(\mathbb R)
    =
    \limsup_n \gamma_n(\{x\le y\})
    \le
    \gamma(\{x\le y\})
    \le
    \mu(\mathbb R),
\]
so $\gamma\in\mathcal A$. Thus $\mathcal A$ is compact for weak convergence. Since $J_\phi$ is lower semicontinuous on $\mathcal A$, if $(\gamma_n)\subset\mathcal A$ satisfies
\[
    J_\phi(\gamma_n)\to\inf_{\eta\in\mathcal A}J_\phi(\eta),
\]
then a weakly convergent subsequence has a limit $\gamma\in\mathcal A$, and lower semicontinuity gives
\[
    J_\phi(\gamma)\le\liminf_n J_\phi(\gamma_n)
    =
    \inf_{\eta\in\mathcal A}J_\phi(\eta).
\]
Hence $\gamma$ attains the minimum.

Let $\gamma$ be any minimizer. Apply
Lemma~\ref{lem:secondary-swap-general-profile} with $\psi=\phi$, and let
$S$ be the full-mass set it provides. Since $\mu\perp\nu$ and
$\gamma\in\mathcal A$, we may take $S$ so that every $(x,y)\in S$ satisfies
$x<y$.

The secondary swap inequality rules out crossing arches. If
$(x,y),(x',y')\in S$ and $x<x'<y<y'$, set
\[
    a=x'-x,\qquad b=y-x',\qquad c=y'-y .
\]
Then the primary costs tie under the rerouting
$(x,y),(x',y')\mapsto (x,y'),(x',y)$, while
Lemma~\ref{lem:strict-concavity-arch-comparison} gives
\[
    \phi(y-x)+\phi(y'-x')
    >
    \phi(y'-x)+\phi(y-x'),
\]
contradicting \eqref{eq:secondary-swap-main}. The same argument after relabeling
rules out the symmetric crossing pattern.

The secondary swap inequality also rules out connected arches. If
$x<y=x'<y'$, set $a=y-x$ and $c=y'-y$. The rerouting to $(x,y')$ and $(y,y)$
ties the primary distance cost. Lemma~\ref{lem:strict-concavity-arch-comparison} gives
\[
    \phi(a)+\phi(c)>\phi(a+c)+\phi(0),
\]
again contradicting \eqref{eq:secondary-swap-main}.

Finally, nested arches have the same orientation because every route in $S$ is
forward. %

By the preceding paragraphs, any strictly concave minimizer $\gamma$ is
concentrated on a full-mass set satisfying Juillet's monotone-arch conditions:
arches do not cross, arches do not connect, and nested arches have the same
orientation. Let $M\coloneqq\mu(\mathbb R)=\nu(\mathbb R)>0$. The full-mass
monotone-arch conditions are unchanged after scaling by $M^{-1}$, and the
completed-graph construction is homogeneous. Juillet's Main Theorem
\cite{juillet2020solution}, applied to the probability coupling
$M^{-1}\gamma$, therefore gives
$M^{-1}\gamma=M^{-1}\gamma_{\mathrm{EC}}$. Hence
$\gamma=\gamma_{\mathrm{EC}}$. Since $\gamma$ was an arbitrary minimizer,
the strictly concave minimizer is unique.

If $\phi$ is merely concave, set
$\phi_\varepsilon(r)=\phi(r)-\varepsilon e^{-r}$. Then $\phi_\varepsilon$ is
finite-valued, strictly concave, and satisfies a lower bound of the same form
$\phi_\varepsilon(r)\ge -C_\varepsilon(1+r)$.
        The strictly concave
        case gives
        \[
            J_{\phi_\varepsilon}(\gamma_{\mathrm{EC}})
            \le
            J_{\phi_\varepsilon}(\gamma)
        \]
        for every $\gamma\in\mathcal A$. Letting $\varepsilon\to0^+$ gives
        $J_\phi(\gamma_{\mathrm{EC}})\le J_\phi(\gamma)$ for every
        $\gamma\in\mathcal A$, since $e^{-|y-x|}$ is bounded and the couplings have
        finite mass. Thus $\gamma_{\mathrm{EC}}$ is a minimizer for $\phi$.
    \end{proof}

\subsection{Proof of Theorem~\ref{thm:6.1'}}
\label{subsec:proof-of-thm-6.1'}

The conclusion of the theorem differs from that of \cite[Theorem 6.1]{ambrosio2003existence} in that the constructed plan $\gamma_{\#}$ is given by the Excursion-Coupling map on each transport ray instead of the monotone map. We follow the same proof methodology: we disintegrate the initial plan $\gamma$ along its transport rays, apply a 1D selection rule on each ray, and then reassemble the result. This disintegration and gluing procedure is enabled by the general measure-theoretic results leveraged by Ambrosio and Pratelli \cite[Thm. 9.1-9.4]{ambrosio2003existence}.
Our key modification is the application of the concave 1D theory (our Theorems \ref{thm:EC-map-gives-a-plan-1D} and \ref{thm:5.2'}) to establish the new conclusion in part (c).

\begin{lemma}[Raywise inheritance]
\label{lem:raywise-inheritance}
Assume that $\mu,\nu,\gamma$, and the $\sigma$-compact diagonal-free set
$\Gamma$ satisfy the hypotheses of Theorem~\ref{thm:6.1'}.  Recall the point-to-open-maximal-ray map $\pi_\Gamma:T_\Gamma\longrightarrow \mathcal S_o(\mathbb R^n)$ and the pair-to-closed-ray map $r:\Gamma\to\mathcal S_c(\mathbb R^n)$ introduced above, equivalently characterized by
\eqref{eq:pair-to-ray-midpoint}.
Set $\sigma\coloneqq r_\#\gamma$, and let
\begin{equation}
    \gamma=\int\gamma_C\,d\sigma(C),\qquad
    \mu_C\coloneqq (\operatorname{proj}_1)_\#\gamma_C,\qquad
    \nu_C\coloneqq (\operatorname{proj}_2)_\#\gamma_C
    \label{eq:ray-disintegration}
\end{equation}
be the disintegration with respect to $r$.  Then, for $\sigma$-a.e. $C$:
\begin{enumerate}
    \item[(i)] $\mu_C$ and $\nu_C$ are both probability measures;
    \item[(ii)] $\mu_C$ is atomless;
    \item[(iii)] $\mu_C\perp\nu_C$;
    \item[(iv)] $\gamma_C$ is forward on the oriented ray and, in an oriented
    coordinate, $F_{\nu_C}\le F_{\mu_C}$;
    \item[(v)] $\mu_C$ and $\nu_C$ have finite first moments. If, in addition,
    \[
        \int\Psi(\|z\|)\,d\mu(z)+
        \int\Psi(\|z\|)\,d\nu(z)<\infty,
        \qquad
        \Psi(r)\coloneqq r\log(1+r),
    \]
    then both conditional $\Psi$-moments are finite for $\sigma$-a.e. $C$.
\end{enumerate}
Finally, $\sigma$ is determined by $(\mu,\Gamma)$, and, relative to this fixed ray measure, the conditional marginals $\mu_C$, $\nu_C$ are determined $\sigma$-a.e. by $\mu$, $\nu$, $\Gamma$, independently of $\gamma$.
\end{lemma}

\begin{proof} When applying \cite[Theorems~9.1--9.3]{ambrosio2003existence}, we use
    the following null-set convention.  Adjoin an isolated label $\dagger$ to
    $\mathcal S_c(\mathbb R^n)$.  A Borel ray map defined on a Borel set carrying
    the entire measure under consideration is extended by the value $\dagger$ off
    that set.  The augmented label space remains locally compact and separable,
    the extension is Borel, and its pushforward gives zero mass to $\dagger$.
    For every genuine ray label, its fiber is therefore unchanged.

    Because $\Gamma$ is diagonal-free, the midpoint map
    $m:\Gamma\to T_\Gamma$, $m(x,y)\coloneqq(x+y)/2$, is continuous.
    By Lemma~\ref{lem:borel-maximal-ray-map}, $\pi_\Gamma$ is Borel, and
    $\operatorname{cl}$ is Borel. Hence
    \eqref{eq:pair-to-ray-midpoint} gives
    $r=\operatorname{cl}\circ\pi_\Gamma\circ m$, so $r$ is Borel.\footnote{\cite{ambrosio2003existence} do not verify that $r$ is Borel, needed for their disintegration theorem (Theorem~9.2). \cite[(18.29)]{maggi2023optimal} prove an analogous map Borel, but only after restricting to pairs whose source already lies in the open transport set; our midpoint construction avoids that restriction.\label{fn:borel-ray-map-gap}}
	\cite[Theorem~9.1]{ambrosio2003existence} therefore gives a Borel kernel
\[
    C\longmapsto\gamma_C\in
    \mathcal P(\mathbb R^n\times\mathbb R^n)
\]
concentrated on $r^{-1}(C)$ for $\sigma$-a.e. $C$.  

The operation of taking a marginal is continuous for weak convergence.
  Therefore $C\mapsto\mu_C,\nu_C$ is Borel. Taking the first and second
  marginals in \eqref{eq:ray-disintegration} gives
  \begin{equation}
      \mu=\int\mu_C\,d\sigma(C),\qquad
      \nu=\int\nu_C\,d\sigma(C).
      \label{eq:conditional-marginal-disintegrations}
  \end{equation}
  Since $\gamma_C$ is a probability measure, so are $\mu_C$ and $\nu_C$,
  which proves (i).

We next prove (ii) by applying
\cite[Theorem~9.4]{ambrosio2003existence} to the
disintegration of $\mu$ along open maximal transport rays. We now verify
the hypotheses of that result.
Since $\gamma$ is concentrated on
$\Gamma$, its first marginal is concentrated on $T_\Gamma^l$; assumption
(ii) of Theorem~\ref{thm:6.1'} and diagonal-freeness imply that $\mu$ is
concentrated on $T_\Gamma$.\footnote{This is the inclusion direction of \cite[Proposition~6.1(ii)]{ambrosio2003existence} and \cite[(18.19)]{maggi2023optimal}, which is all we use; the accompanying disjointness clauses in \cite[Remark~18.4]{maggi2023optimal} fail for general transport supports and are not needed here, since diagonal-freeness already makes the motionless set empty.\label{fn:motionless-set-gap}}  Put
\[
    \bar\pi_\Gamma\coloneqq \operatorname{cl}\circ\pi_\Gamma.
\]
For $\gamma$-a.e. $(x,y)$, one has
$r(x,y)=\bar\pi_\Gamma(x) \supseteq [[x,y]]$ and $x \prec_{r(x,y)} y$.  Indeed, if $x$ lies in the interior of a
witnessing transport ray and the ray containing $]]x,y[[$ had a different
orientation, condition~\eqref{eq:18} would force a common source or common
target; the first is impossible because $x$ is interior to the witnessing
ray, while the common-target case contradicts the assumed
difference in orientation. Thus the orientations agree. Since the two segments intersect, they belong to the same maximal transport ray.  Consequently
\[
    \sigma=(\bar\pi_\Gamma)_\#\mu,
\]
and the family $\mu_C$ in \eqref{eq:ray-disintegration} is a disintegration of
$\mu$ with respect to $\bar\pi_\Gamma$.

Let $B\coloneqq \bigcup_hK_h$, with the compact sets $K_h$ from assumption (iii).
Then $B$ is Borel and has full $\mu$-measure.  On $B$, the map
$\pi_\Gamma$ is Borel, $x\in\pi_\Gamma(x)$, distinct open maximal rays are
disjoint, and their direction map is countably Lipschitz because
$\tau_\Gamma|_{K_h}$ is Lipschitz.  Moreover $\mu\ll\mathcal L^n$.
Thus Theorem~9.4 and Remark~9.1 of \cite{ambrosio2003existence}, applied to
$\pi_\Gamma|_B$, give conditional source measures absolutely continuous with
respect to $\mathcal H^1$ on the open rays.  Since the closure map is
one-to-one, the uniqueness theorem for disintegrations
\cite[Theorem~9.2]{ambrosio2003existence} identifies those conditionals with
the $\mu_C$ above.  Hence $\mu_C\ll\mathcal H^1|_C$ and is atomless for
$\sigma$-a.e. $C$, proving (ii).

For (iii), choose a Borel separator $A\subset\mathbb R^n$ with
$\mu(A^c)=0$ and $\nu(A)=0$.  By
\eqref{eq:conditional-marginal-disintegrations} and nonnegativity,
\[
    0=\int\mu_C(A^c)\,d\sigma(C),\qquad
    0=\int\nu_C(A)\,d\sigma(C).
\]
Thus $\mu_C(A^c)=\nu_C(A)=0$ for a.e. $C$, so
$\mu_C\perp\nu_C$.

Fix such a ray $C$, choose any $z_C\in C$, and let $v_C$ be its oriented unit
direction.  The coordinate
\[
    s_C:C\to\mathbb R,\qquad s_C(z)\coloneqq \langle z-z_C,v_C\rangle
\]
identifies $C$ with an interval of $\mathbb R$ and, for $z,z'\in C$, satisfies
\[
    |s_C(z)-s_C(z')|=\|z-z'\|,\qquad
    z\preceq_C z'\ \Longleftrightarrow\ s_C(z)\leq s_C(z')\, .
\]
  Since $[[x,y]] \subseteq r(x,y)$ for every $(x,y) \in \Gamma$ and $x \prec_{r(x,y)} y$ holds for $\gamma$-a.e. $(x,y)$, disintegrating and using
  diagonal-freeness gives
  \[
      \gamma_C\bigl(\{(x,y)\in C\times C:
      s_C(x)<s_C(y)\}\bigr)=1
  \]
  for $\sigma$-a.e. $C$. Thus $\gamma_C$ is forward.

  Write $F_{\mu_C}$ and $F_{\nu_C}$ for the distribution functions of
  $(s_C)_\#\mu_C$ and $(s_C)_\#\nu_C$, respectively. For every $t\in\mathbb R$,
  \[
      F_{\nu_C}(t)
      =\int\mathbf 1_{\{s_C(y)\le t\}}\,d\gamma_C
      \le
      \int\mathbf 1_{\{s_C(x)\le t\}}\,d\gamma_C
      =F_{\mu_C}(t),
  \]
  which proves (iv).

Tonelli's theorem and
\eqref{eq:conditional-marginal-disintegrations} give
\[
    \int\!\left(\int\|z\|\,d\mu_C(z)\right)d\sigma(C)
    =\int\|z\|\,d\mu(z)<\infty,
\]
and the same identity for $\nu$.  Thus the conditional Euclidean first
moments are finite a.e.; also
$|s_C(z)|= \|z-z_C\| \le \|z\|+\|z_C\|$, so the one-dimensional coordinate moments are
finite.  This proves the first part of (v).  Replacing $\|z\|$ by the nonnegative function
$\Psi(\|z\|)$ in the same Tonelli argument proves the second part of (v).

It remains to show that $\sigma$ is determined by $\mu$ and $\Gamma$, and that the conditional marginals $\mu_C$, $\nu_C$ are determined $\sigma$-a.e. by $\mu$, $\nu$, $\Gamma$, independently of $\gamma$.
Put $\bar\pi_\Gamma\coloneqq\operatorname{cl}\circ\pi_\Gamma$. Since
$r(x,y)=\bar\pi_\Gamma(x)$ for $\gamma$-a.e. $(x,y)$, we have
\[
    \sigma=(\bar\pi_\Gamma)_\#\mu \, ,
\]
and therefore $\sigma$ is determined by $\mu$ and $\Gamma$.
Moreover, the conditional source measures $\mu_C$ are consequently the
disintegration of $\mu$ with respect to $\bar\pi_\Gamma$ and are $\sigma$-a.e. canonical
for fixed $\mu,\Gamma$.

For the target mass in $T_\Gamma$, the above implies that
$r(x,y)=\bar\pi_\Gamma(y)$ whenever $y\in T_\Gamma$. Let
$\sigma_\nu\coloneqq(\bar\pi_\Gamma)_\#(\nu|_{T_\Gamma})$. 
The existence of the input plan implies $\sigma_\nu\le\sigma$.
     Indeed, for every Borel ray set $R$,
     \[
         \sigma_\nu(R)
         =
         \gamma\{(x,y):y\in T_\Gamma,\ r(x,y)\in R\}
         \le \sigma(R).
     \]
     Consequently, $h\coloneqq d\sigma_\nu/d\sigma$ has a Borel version
     satisfying $0\le h\le1$.

Disintegrate $\nu|_{T_\Gamma}$ over $\sigma_\nu$ and multiply its
probability conditionals by $d\sigma_\nu/d\sigma$; denote the resulting
subprobability measures by $\nu_C^{\mathrm{int}}$. Equivalently, they are
characterized by
\[
    \int_R\nu_C^{\mathrm{int}}(E)\,d\sigma(C)
    =
    \nu\bigl(E\cap T_\Gamma\cap\bar\pi_\Gamma^{-1}(R)\bigr)
\]
for Borel $E\subset\mathbb R^n$ and Borel ray sets $R$. 
Uniqueness of
disintegration identifies $\nu_C|_{T_\Gamma}$ with
$\nu_C^{\mathrm{int}}$ for $\sigma$-a.e. $C$.

If a target point of a pair in $\Gamma$ lies outside $T_\Gamma$, it is the
upper endpoint $e^+(C)$ of its oriented maximal ray. This case cannot occur
on a ray unbounded in the forward direction. Since $\mu_C$ and $\nu_C$ are
probabilities, mass balance therefore forces
\[
    \nu_C
    =
    \nu_C^{\mathrm{int}}
    +
    \bigl(1-\nu_C^{\mathrm{int}}(\mathbb R^n)\bigr)\delta_{e^+(C)},
\]
with the endpoint term understood as zero when $C$ has no upper
endpoint. This proves the final assertion.

\end{proof}

\begin{lemma}[Measurable dependence of the raywise EC plans]
\label{lem:measurable-raywise-ec}
In the setting of Lemma~\ref{lem:raywise-inheritance}, one can define raywise
transformed plans $\gamma_{C,\#}$ such that
\[
    C\longmapsto\gamma_{C,\#}
\]
is a Borel map into the fixed Polish space
$\mathcal P(\mathbb R^n\times\mathbb R^n)$, and for $\sigma$-a.e. $C$ it is
the unique excursion coupling between $\mu_C$ and $\nu_C$ on the oriented
ray $C$.\footnotemark
\end{lemma}
\footnotetext{\cite{ambrosio2003existence} assert the analogous measurable dependence without proof, citing their Appendix. \cite[Theorem~18.7(iii)]{maggi2023optimal} prove it rigorously, but via a closed-form formula specific to the monotone coupling; no such formula exists for the excursion coupling, so we prove this directly below.\label{fn:measurable-gluing-gap}}

\begin{proof}
Discard a Borel $\sigma$-null set so that all conclusions of
Lemma~\ref{lem:raywise-inheritance} hold, and denote the remaining standard
Borel ray space by $Y_0$.  Put $E\coloneqq\mathbb R^n\times\mathbb R^n$.  For
$C\in Y_0$, let
\[
    R_C\coloneqq\{(x,y)\in C\times C:x\preceq_C y\}
\]
be the forward relation and define
\[
    K(C)\coloneqq \left\{\eta\in\mathcal P(E):
    (\operatorname{proj}_1)_\#\eta=\mu_C,\quad
    (\operatorname{proj}_2)_\#\eta=\nu_C,\quad
    \eta(R_C)=1\right\}.
\]
This is precisely the set of forward couplings of the two conditional
marginals. 

For each $C$, the set $K(C)$ is non empty and compact. 
Nonemptiness follows
from $\gamma_C\in K(C)$.  Couplings with the two fixed probability marginals
form a tight, narrowly closed, hence compact subset of $\mathcal P(E)$; since
$R_C$ is closed, the support condition is closed as well.

The graph of the multifunction $K$ is Borel.  Indeed, the marginal maps
\[
    \eta\longmapsto(\operatorname{proj}_i)_\#\eta
\]
are continuous and $C\mapsto\mu_C,\nu_C$ is Borel as constructed in Lemma~\ref{lem:raywise-inheritance}.

For $R\in\mathbb N$, set
  \[
      \overline B_R\coloneqq\{z\in\mathbb R^n:\|z\|\le R\}.
  \]
  The closure correspondence transports the metric
  $d_{\mathrm{AP}}$ defined above in \eqref{eq:AP-metric} to
  $\mathcal S_c(\mathbb R^n)$, with the same notation. For a closed oriented
  ray $C$, write
  $C\cap\overline B_R=[x_R(C),y_R(C)]$, using the same oriented-order,
  empty-intersection, and singleton conventions as above.
  In this metric the joint relation
  \[
      \mathcal R\coloneqq\{(C,x,y):C\in Y_0,\ (x,y)\in R_C\}
  \]
  is Borel.  Indeed, suppose
  $(C_j,x_j,y_j)\to(C,x,y)$ with $(x_j,y_j)\in R_{C_j}$.  Choose $R$ so large
  that $x,y$ and, eventually, $x_j,y_j$ lie in $\overline B_R$.  The ordered endpoints
  of $C_j\cap \overline B_R$ converge to those of $C\cap \overline B_R$; membership of $x_j,y_j$
  and their order between these endpoints therefore pass to the limit.
  Thus the corresponding relation in
  $\mathcal S_c(\mathbb R^n)\times E$ is closed, and its restriction to the
  Borel set $Y_0\times E$ is Borel.

  The standard parameterized-integration
  lemma, proved first for rectangles and then by a monotone-class argument,
  shows that
\[
    (C,\eta)\longmapsto
    \eta(R_C^c)
    =
    \int_E\mathbf 1_{\mathcal R^c}(C,x,y)\,d\eta(x,y)
\]
  is Borel.  Thus the two marginal equalities and the equation
  $\eta(R_C^c)=0$ define a Borel graph for $K$.

  A Borel graph is not by itself the weak measurability hypothesis in the
  measurable minimum theorem, so we record the missing projection step.  If
  $U\subset\mathcal P(E)$ is open, then
  \[
      \operatorname{Gr}K\cap(Y_0\times U)
  \]
  is Borel, and its section at $C$ is $K(C)\cap U$.  This is a
  $K_\sigma$ set (countable union of compact sets) because it is an open subset of the compact metric space
  $K(C)$.  The Arsenin--Kunugui theorem
  \cite[Theorem~18.18]{kechris1995classical} therefore makes its
  projection
  \[
      \{C\in Y_0:K(C)\cap U\ne\varnothing\}
  \]
  Borel.  Hence $K$ is weakly measurable.

Use the fixed selector profile
\[
    \chi(r)\coloneqq1-e^{-r},\qquad r\ge0.
\]
It is bounded, continuous, and strictly concave.  Therefore
\[
    J_\chi(\eta)\coloneqq\int_E\chi(\|x-y\|)\,d\eta(x,y)
\]
  is a bounded continuous function on $\mathcal P(E)$.  The measurable minimum
  theorem \cite[Theorem~18.19]{aliprantisborder2006} now applies to the weakly
  measurable multifunction $K$ with nonempty compact values
and the Borel, fiberwise continuous objective $J_\chi$: the value
\[
    v(C)\coloneqq\min_{\eta\in K(C)}J_\chi(\eta)
\]
is Borel, and
\[
    \operatorname{Gr}M
    \coloneqq
    \{(C,\eta):\eta\in K(C),\ J_\chi(\eta)=v(C)\}
\]
is Borel with nonempty compact sections.  Equivalently, the theorem's
selection part (or Borel uniformization for compact sections) supplies a
Borel selector $C\mapsto\eta_C^*\in M(C)$.  These facts audit the hypotheses
of the measurable-minimum invocation: one fixed Polish decision space, a
Borel feasible graph with compact fibers, and a bounded continuous objective.

After the isometric oriented-coordinate identification from
Lemma~\ref{lem:raywise-inheritance}, Theorem~\ref{thm:5.2'} applies to
$\chi$ and says that $M(C)$ is a singleton, equal to the excursion coupling.
Consequently $\eta_C^*=\gamma_{C,\#}$ for $\sigma$-a.e. $C$, proving the
claimed Borel dependence.
\end{proof}

\begin{proof}[Proof of Theorem~\ref{thm:6.1'}]
We follow the first four steps of the proof of
\cite[Theorem~6.1]{ambrosio2003existence}: disintegration, atomlessness,
raywise rearrangement, and gluing.  The fifth, fixed-point step is absent
because $\Gamma$ is diagonal-free.

\textbf{Step 1. Disintegration.}
Fix the given $\gamma$ and $\Gamma$; all ray conditionals in the construction
are those in \eqref{eq:ray-disintegration} from Lemma~\ref{lem:raywise-inheritance}. Moreover, the lemma guarantees that $\sigma$ is determined by $\mu$ and $\Gamma$, and that the conditional marginals $\mu_C$, $\nu_C$ are determined $\sigma$-a.e. by $\mu$, $\nu$, $\Gamma$, independently of $\gamma$. 

\textbf{Step 2. Atomless conditional measures.}
By
Lemma~\ref{lem:raywise-inheritance}, for $\sigma$-a.e. $C$ the conditional
marginals are mutually singular probability measures, the source is
atomless, the target stochastically dominates the source in the oriented
coordinate, and both have finite first moments.

\textbf{Step 3. Applying the 1D theory.}
Theorems
\ref{thm:EC-map-gives-a-plan-1D} and \ref{thm:5.2'} therefore give a Borel
map $T_{C,\#}:C\to C$ and its graph plan
\[
    \gamma_{C,\#}\coloneqq(\mathrm{Id},T_{C,\#})_\#\mu_C
    \in\Pi(\mu_C,\nu_C)\, .
\]
For every finite-valued concave $\phi$ satisfying
$\phi(d)\ge-K_\phi(1+d)$,
\begin{equation}
    \int_{C\times C}\phi(\|x-y\|)\,d\gamma_{C,\#}
    \le
    \int_{C\times C}\phi(\|x-y\|)\,d\gamma_C.
    \label{ineq:c_on_MTR}
\end{equation}
When $\phi$ is strictly concave, equality in
\eqref{ineq:c_on_MTR} holds if and only if
$\gamma_C=\gamma_{C,\#}$.  Both plans are forward, so the ray coordinate
also gives
\begin{equation}
    \int_{C\times C}\|x-y\|\,d\gamma_{C,\#}
    =
    \int_{C\times C}\|x-y\|\,d\gamma_C\, .
    \label{eq:raywise-distance-equality}
\end{equation}

Uniqueness of $\gamma_{C,\#}$ from $(\mu_C,\nu_C)$ will make the glued plan we construct in Step 4 independent of $\gamma$.

\textbf{Step 4. Gluing the maps.}
Lemma~\ref{lem:measurable-raywise-ec} proves that
$C\mapsto\gamma_{C,\#}$ is Borel as a
$\mathcal P(\mathbb R^n\times\mathbb R^n)$-valued map.  Choose the Borel
representative $T_{C,\#}$ from Step~3 on a fixed Borel set of full
$\sigma$-measure, and make arbitrary choices on the exceptional rays.  Since
the source conditionals form the point-to-ray disintegration identified in
Lemma~\ref{lem:raywise-inheritance},
\[
    \mu_C\bigl(\bar\pi_\Gamma^{-1}(C)\bigr)=1
\]
for $\sigma$-a.e. $C$.  Define one set-theoretic map
$f:\mathbb R^n\to\mathbb R^n$ by
\[
    f(x)\coloneqq T_{\bar\pi_\Gamma(x),\#}(x)
    \quad\text{for }x\in T_\Gamma,
\]
and give $f$ an arbitrary fixed value outside $T_\Gamma$.  For
$\sigma$-a.e. $C$, its restriction to
$\bar\pi_\Gamma^{-1}(C)$ agrees with $T_{C,\#}$ and is therefore
$\mu_C$-measurable.  Moreover,
\[
    (\mathrm{Id},f)_\#\mu_C
    =
    (\mathrm{Id},T_{C,\#})_\#\mu_C
    =
    \gamma_{C,\#}\, .
\]
Thus both conditions 
\cite[Theorem~9.3]{ambrosio2003existence} hold.  That theorem gives a global
Borel representative $t_\#:\mathbb R^n\to\mathbb R^n$ satisfying
$t_\#=T_{C,\#}$ $\mu_C$-a.e. for $\sigma$-a.e. $C$.
Set
\[
    \gamma_\#\coloneqq(\mathrm{Id},t_\#)_\#\mu=\gamma_\#^\Gamma\, .
\]
Disintegration, or equivalently \cite[(37)]{ambrosio2003existence}, yields
\begin{equation}
    \gamma_\#=\int\gamma_{C,\#}\,d\sigma(C).
    \label{eq:glued-ec-disintegration}
\end{equation}
Taking second marginals in this identity gives
\[
    (\operatorname{proj}_2)_\#\gamma_\#
    =
    \int\nu_C\,d\sigma(C)
    =
    \nu\, ,
\]
while its first marginal is $\mu$ by construction.  Thus
$\gamma_\#\in\Pi(\mu,\nu)$ and is induced by the Borel map $t_\#$.
Moreover, \eqref{eq:glued-ec-disintegration} is concentrated ray by ray on
$C\times C$, proving conclusion (b).

Integrating \eqref{eq:raywise-distance-equality} gives
\begin{equation}
    \int\|x-y\|\,d\gamma_\#
    =
    \int\|x-y\|\,d\gamma.
    \label{eq:gammahash-same-as-gamma-on-C1}
\end{equation}
If $\Gamma$ is $c$-cyclically monotone for
$c(x,y)=\|x-y\|$, the sufficiency part of
\cite[Theorem~3.2]{ambrosio2003existence}, applied after normalization by
the common total mass, shows that $\gamma$ is distance-optimal; the displayed
identity then shows that $\gamma_\#$ is distance-optimal as well.  This proves
conclusion (a).

Finally, finite-valued concavity gives an affine upper bound for $\phi$, while
the hypothesis gives an affine lower bound.  The global first moments
therefore make all the following integrals finite and justify disintegration.
Integrating \eqref{ineq:c_on_MTR} gives
\[
    \int\phi(\|x-y\|)\,d\gamma_\#
    \le
    \int\phi(\|x-y\|)\,d\gamma,
\]
which is conclusion (c).  If $\phi$ is strictly concave and global equality
holds, the nonnegative measurable raywise gaps
\[
    D(C)\coloneqq
    \int\phi(\|x-y\|)\,d\gamma_C
    -
    \int\phi(\|x-y\|)\,d\gamma_{C,\#}
\]
have $\int D\,d\sigma=0$.  Hence $D(C)=0$ for $\sigma$-a.e. $C$, and
raywise uniqueness gives $\gamma_C=\gamma_{C,\#}$ a.e.  Comparing
\eqref{eq:ray-disintegration} and
\eqref{eq:glued-ec-disintegration} then yields
$\gamma=\gamma_\#$.  %
\end{proof}


\section{Power costs and the logarithmic secondary problem}
\label{sec:power-costs}

For the power costs
\[
    c_\epsilon(x,y)=\|x-y\|^{1-\epsilon},
    \qquad 0<\epsilon<1,
\]
the first-order profile at $\epsilon=0$ is
\[
    g(r)\coloneqq -r\log r,\qquad r\ge0,
\]
where throughout this section we use the convention $0\log0=0$.  The
negative part of $g$ grows faster than linearly.  We therefore use the
Young function
\[
    \Psi(r)\coloneqq r\log(1+r),\qquad r\ge0,
\]
and write
\[
    \mathcal M_{1+}(\mathbb R^n)
    \coloneqq 
    \left\{\eta\in\mathcal M_+(\mathbb R^n):
    \eta(\mathbb R^n)<\infty,\quad
    \int_{\mathbb R^n}\Psi(\|x\|)\,d\eta(x)<\infty\right\}.
\]
Thus $\mathcal P_{1+}(\mathbb R^n)
\coloneqq \mathcal M_{1+}(\mathbb R^n)\cap\mathcal P(\mathbb R^n)$.

\begin{lemma}[Two-point Orlicz estimate]
\label{lem:power-two-point-orlicz}
There is a universal constant $C_\Psi$ such that, for all
$x,y\in\mathbb R^n$,
\[
    \Psi(\|x-y\|)
    \le
    C_\Psi\bigl(1+\Psi(\|x\|)+\Psi(\|y\|)\bigr).
\]
\end{lemma}

\begin{proof}
The function $\Psi$ is increasing and convex.  If $a,b\ge0$, convexity gives
\[
    \Psi(a+b)
    \le \frac12\Psi(2a)+\frac12\Psi(2b).
\]
Moreover,
\[
    \Psi(2t)
    =2t\log(1+2t)
    \le 2\Psi(t)+2t\log2.
\]
For $t\le1$ we have $t\log2\le\log2$, while for $t\ge1$ we have
$t\log2\le\Psi(t)$.  Hence
\[
    \Psi(2t)\le4\Psi(t)+2\log2
\]
for every $t\ge0$.  Taking $a=\|x\|$, $b=\|y\|$, and using
$\|x-y\|\le a+b$ proves the claim, for example with $C_\Psi=2$.
\end{proof}

\begin{lemma}[Uniform integrability for fixed marginals]
\label{lem:power-fixed-marginal-ui}
Let $\mu,\nu\in\mathcal M_{1+}(\mathbb R^n)$ have equal mass.  Then
\[
    (x,y)\longmapsto\Psi(\|x-y\|)
\]
is uniformly integrable over $\Pi(\mu,\nu)$.  Consequently
$|g(\|x-y\|)|$ is uniformly integrable over $\Pi(\mu,\nu)$, and both
functions have uniformly bounded integrals on that coupling class.
\end{lemma}

\begin{proof}
We give the marginal-tail argument explicitly.  Put
$a=\|x\|$, $b=\|y\|$,
\[
    H_\mu(x)\coloneqq\Psi(2\|x\|),\qquad
    H_\nu(y)\coloneqq\Psi(2\|y\|).
\]
The scaling estimate in the preceding proof shows that
$H_\mu\in L^1(\mu)$ and $H_\nu\in L^1(\nu)$.  Also,
\[
    \Psi(\|x-y\|)
    \le\Psi(a+b)
    \le\frac12\bigl(H_\mu(x)+H_\nu(y)\bigr).
\]
Let $A_R\coloneqq\{(x,y):\|x-y\|>R\}$.  Since
\[
    A_R\subset\{a>R/2\}\cup\{b>R/2\},
\]
the integral of $\Psi(\|x-y\|)$ over $A_R$ is bounded by one half of the
sum of four terms.  Two of them are the marginal tails
\[
    \int_{\{a>R/2\}}H_\mu\,d\mu,
    \qquad
    \int_{\{b>R/2\}}H_\nu\,d\nu.
\]
For either cross term, and every $L>0$, the fixed marginals give
\begin{align*}
    \int_{\{b>R/2\}}H_\mu(x)\,d\gamma(x,y)
    &\le
    \int_{\{H_\mu>L\}}H_\mu\,d\mu
    +L\,\nu(\{b>R/2\}),\\
    \int_{\{a>R/2\}}H_\nu(y)\,d\gamma(x,y)
    &\le
    \int_{\{H_\nu>L\}}H_\nu\,d\nu
    +L\,\mu(\{a>R/2\}).
\end{align*}
These bounds hold for every $\gamma\in\Pi(\mu,\nu)$.  First choose $L$
so that the two $H$-tails are small, and then choose $R$ so that the
remaining marginal tails and probabilities are small.  It follows that
\[
    \lim_{R\to\infty}
    \sup_{\gamma\in\Pi(\mu,\nu)}
    \int_{A_R}\Psi(\|x-y\|)\,d\gamma=0.
\]
Since $\Psi$ is increasing and tends to infinity, this is precisely the
tail criterion for uniform integrability.  Notice that the argument uses
the fixed marginal tails, not merely a uniform $L^1$ bound.

Finally,
\[
    |g(r)|\le\frac1e\quad(0\le r\le1),
    \qquad
    |g(r)|=r\log r\le\Psi(r)\quad(r\ge1).
\]
The uniform-integrability assertion for $|g(\|x-y\|)|$ follows from the
one just proved.  Uniform boundedness of the integrals follows as well.
\end{proof}

For $\gamma\in\Pi(\mu,\nu)$ define the logarithmic secondary functional
\[
    J(\gamma)\coloneqq\int_{\mathbb R^n\times\mathbb R^n}
    g(\|x-y\|)\,d\gamma(x,y).
\]

\begin{lemma}[Continuity of the logarithmic secondary functional]
\label{lem:power-J-continuity}
Under the hypotheses of Lemma~\ref{lem:power-fixed-marginal-ui},
$J$ is finite-valued and weakly continuous on $\Pi(\mu,\nu)$.  In
particular, it is weakly lower semicontinuous.
\end{lemma}

\begin{proof}
For $M>0$ set
\[
    g_M(r)\coloneqq\max\{g(r),-M\}.
\]
Since $g$ is continuous and bounded above by $1/e$, $g_M$ is bounded and
continuous.  Moreover, uniformly over $\gamma\in\Pi(\mu,\nu)$,
\[
    0\le
    \int\bigl(g_M(\|x-y\|)-g(\|x-y\|)\bigr)\,d\gamma
    \le
    \int_{\{|g(\|x-y\|)|>M\}}|g(\|x-y\|)|\,d\gamma,
\]
and the right-hand side tends to zero by
Lemma~\ref{lem:power-fixed-marginal-ui}.  If $\gamma_k$ converges weakly to
$\gamma$, then for fixed $M$,
\[
    \int g_M(\|x-y\|)\,d\gamma_k
    \longrightarrow
    \int g_M(\|x-y\|)\,d\gamma.
\]
Letting $M\to\infty$ uniformly in $k$ proves $J(\gamma_k)\to J(\gamma)$.
\end{proof}

\begin{lemma}[Power difference quotients]
\label{lem:power-quotient-bounds}
For $0<\epsilon\le1/2$, define
\[
    q_\epsilon(r)\coloneqq\frac{r^{1-\epsilon}-r}{\epsilon},
    \qquad r\ge0.
\]
Then
\[
    0\le q_\epsilon(r)\le\frac2e
    \quad(0\le r\le1),
    \qquad
    |q_\epsilon(r)|\le r\log r
    \quad(r\ge1).
\]
In addition,
\[
    q_\epsilon(r)\ge g(r)\quad\text{for all }r\ge0,
    \qquad
    q_\epsilon(r)\longrightarrow g(r)
    \quad\text{as }\epsilon\to0^+.
\]
\end{lemma}

\begin{proof}
For $0<r\le1$, the mean value theorem in the variable $\epsilon$ gives,
for some $\theta\in]0,\epsilon[$,
\[
    q_\epsilon(r)=r^{1-\theta}|\log r|
    \le r^{1/2}|\log r|\le\frac2e.
\]
For $r\ge1$, the same argument gives
\[
    |q_\epsilon(r)|=r^{1-\theta}\log r\le r\log r.
\]
The assertions at $r=0$ follow from the stated convention.  Finally,
$e^s\ge1+s$, applied with $s=-\epsilon\log r$, yields
\[
    q_\epsilon(r)
    =r\,\frac{e^{-\epsilon\log r}-1}{\epsilon}
    \ge-r\log r=g(r).
\]
The pointwise limit follows by differentiating $r^{1-\epsilon}$ at
$\epsilon=0$.
\end{proof}

\begin{prop}[Gamma limits and power-cost selection]
\label{prop:power-gamma-selection}
Let $\mu,\nu\in\mathcal M_{1+}(\mathbb R^n)$ be mutually singular finite
positive Borel measures with equal positive mass, and assume
$\mu\ll\mathcal L^n$.  On $\Pi(\mu,\nu)$, with its weak topology, define
\[
    d(x,y)\coloneqq\|x-y\|,\qquad
    F_\epsilon(\gamma)\coloneqq\int d^{1-\epsilon}\,d\gamma,\qquad
    F(\gamma)\coloneqq\int d\,d\gamma,
\]
and
\[
    m\coloneqq\min_{\eta\in\Pi(\mu,\nu)}F(\eta),\qquad
    F'_\epsilon(\gamma)\coloneqq\frac{F_\epsilon(\gamma)-m}{\epsilon}.
\]
Keep $J$ as the finite-valued functional defined above, and define the
extended functional
\[
    F'(\gamma)\coloneqq
    \begin{cases}
        J(\gamma),&\gamma\in\Pi_1(\mu,\nu),\\
        +\infty,&\gamma\notin\Pi_1(\mu,\nu).
    \end{cases}
\]
Then
\[
    F_\epsilon\xrightarrow{\Gamma}F,
    \qquad
    F'_\epsilon\xrightarrow{\Gamma}F'
\]
as $\epsilon\to0^+$.  The two families $F_\epsilon$ and $F'_\epsilon$ are equicoercive.  Consequently,
every weak limit point of minimizers of $F_\epsilon$ belongs to
$\Pi_1(\mu,\nu)$ and minimizes $J$ there.  This secondary minimum is
attained and finite.
\end{prop}

\begin{proof}
The fixed-marginal coupling class is weakly compact.  Indeed, it is tight
because both marginals are fixed finite measures, and it is weakly closed.
Also, the log moment implies a finite first moment.

We first record that $F$ is weakly continuous on this coupling class.
The functions $d\wedge R$ are bounded and continuous, while
\[
    \sup_{\gamma\in\Pi(\mu,\nu)}
    \int_{\{d>R\}}d\,d\gamma
    \le
    \frac{1}{\log(1+R)}
    \sup_{\gamma\in\Pi(\mu,\nu)}\int\Psi(d)\,d\gamma
    \xrightarrow{R\to\infty}0 \, ,
\]
using Lemma~\ref{lem:power-fixed-marginal-ui}.
Thus bounded truncation proves the claimed continuity.  For
$0<\epsilon\le1/2$, Lemma~\ref{lem:power-quotient-bounds} gives
\[
    |q_\epsilon(d)|\le\frac2e+\Psi(d).
\]
The right-hand side has a uniformly bounded integral over
$\Pi(\mu,\nu)$, and therefore
\[
    \sup_{\gamma\in\Pi(\mu,\nu)}
    |F_\epsilon(\gamma)-F(\gamma)|
    \le
    \epsilon\sup_{\gamma\in\Pi(\mu,\nu)}
    \int|q_\epsilon(d)|\,d\gamma
    \longrightarrow0.
\]
Hence $F_\epsilon$ converges uniformly to the continuous functional $F$,
which proves both inequalities in the first Gamma limit, including the
constant recovery sequence.

We turn to the rescaled functionals.  First note that
\[
    F'_\epsilon(\gamma)
    =
    \frac{F(\gamma)-m}{\epsilon}
    +\int q_\epsilon(d)\,d\gamma.
\]
Since $F(\gamma)\ge m$ and $q_\epsilon\ge g$ by Lemma~\ref{lem:power-quotient-bounds}, the functionals
$F'_\epsilon$ have a common finite lower bound by
Lemma~\ref{lem:power-fixed-marginal-ui}.

Let $\gamma_\epsilon\to\gamma$ weakly.  We spell out the finite-liminf
step.  If
\[
    \liminf_{\epsilon\to0^+}F'_\epsilon(\gamma_\epsilon)<+\infty,
\]
the common lower bound lets us choose a subsequence $\epsilon_k\to0$ on
which the displayed liminf is a finite real number.  Along this subsequence,
\[
    F_{\epsilon_k}(\gamma_{\epsilon_k})
    =m+\epsilon_kF'_{\epsilon_k}(\gamma_{\epsilon_k})
    \longrightarrow m.
\]
The first Gamma-liminf inequality gives
\[
    F(\gamma)\le m.
\]
By the definition of $m$, equality holds, so
$\gamma\in\Pi_1(\mu,\nu)$.  Furthermore, using Lemma~\ref{lem:power-quotient-bounds},
\[
    F'_\epsilon(\gamma_\epsilon)
    \ge\int g(d)\,d\gamma_\epsilon=J(\gamma_\epsilon).
\]
Lemma~\ref{lem:power-J-continuity} now yields
\[
    \liminf_{\epsilon\to0^+}F'_\epsilon(\gamma_\epsilon)
    \ge J(\gamma)=F'(\gamma).
\]
If the liminf is infinite the inequality is automatic.  In particular,
when $\gamma\notin\Pi_1(\mu,\nu)$, a finite liminf is impossible and the
Gamma-liminf is $+\infty=F'(\gamma)$.

For the recovery sequence, let $\gamma\in\Pi_1(\mu,\nu)$ and take the
constant sequence.  Then
\[
    F'_\epsilon(\gamma)=\int q_\epsilon(d)\,d\gamma.
\]
The quotient converges pointwise to $g(d)$, and
\[
    |q_\epsilon(d)|\le\frac2e+\Psi(d)
\]
for $0<\epsilon\le1/2$.  This is an integrable dominating function for the
fixed plan $\gamma$, so dominated convergence gives
\[
    F'_\epsilon(\gamma)\longrightarrow J(\gamma)=F'(\gamma).
\]
At points where $F'=+\infty$, the Gamma-limsup requirement is automatic.
This proves the second Gamma limit.

Compactness of $\Pi(\mu,\nu)$ makes both families equicoercive.  Also
$\Pi_1(\mu,\nu)=\{F=m\}$ is nonempty and compact, and $J$ is continuous
and finite there.  Hence its minimum is attained and is a real number.

Finally, let $\gamma_\epsilon$ minimize $F_\epsilon$, equivalently
$F'_\epsilon$, and let a subsequence converge weakly to $\gamma_0$.  For
every $\eta\in\Pi_1(\mu,\nu)$, minimality and the constant recovery
sequence give
\[
    F'(\gamma_0)
    \le\liminf_{\epsilon\to0^+}F'_\epsilon(\gamma_\epsilon)
    \le\limsup_{\epsilon\to0^+}F'_\epsilon(\eta)
    =J(\eta).
\]
Thus $\gamma_0\in\Pi_1(\mu,\nu)$ and it minimizes $J$ on that class.
\end{proof}

\begin{lemma}[Orlicz lexicographic variational principle]
\label{lem:power-orlicz-lexicographic}
Let $\mu,\nu\in\mathcal M_{1+}(\mathbb R)$ have equal mass.  Let
$c_2:\mathbb R^2\to\mathbb R$ be lower semicontinuous and suppose
\[
    c_2(x,y)\ge
    -C\bigl(1+\Psi(|x|)+\Psi(|y|)\bigr)
\]
for some $C\ge0$.  Suppose that $\gamma$ minimizes
$\int c_2\,d\eta$ among $\eta\in\Pi_1(\mu,\nu)$ and that its value is
finite.  Then there are a 1-Lipschitz function $u$ and a full
$\gamma$-mass Borel set
\[
    S\subset\Gamma_u
    \coloneqq\{(x,y):|x-y|=u(x)-u(y)\}
\]
such that, for every finite family $(x_i,y_i)_{i=1}^N\subset S$ and every
permutation $\tau$,
\[
    \sum_i|x_i-y_i|
    \le\sum_i|x_i-y_{\tau(i)}|.
\]
Whenever equality holds, one also has
\[
    \sum_i c_2(x_i,y_i)
    \le\sum_i c_2(x_i,y_{\tau(i)}).
\]
\end{lemma}

\begin{proof}
Choose $u$ and $\Gamma_u$ from
Lemma~\ref{lem:linear-cost-contact-set}.  That lemma identifies
$\Pi_1(\mu,\nu)$ with the couplings concentrated on $\Gamma_u$.  Encode
this restriction and the lower control in the extended cost
\[
    \widehat c_2(x,y)\coloneqq
    \begin{cases}
      c_2(x,y)+C\bigl(1+\Psi(|x|)+\Psi(|y|)\bigr),
          &(x,y)\in\Gamma_u,\\
      +\infty,&(x,y)\notin\Gamma_u.
    \end{cases}
\]
The contact set is closed, so $\widehat c_2$ is lower semicontinuous; the
assumed lower bound on $c_2$ makes it nonnegative.  On couplings concentrated on
$\Gamma_u$, the added term has the fixed finite integral
\[
    C\left[
    \mu(\mathbb R)+\int\Psi(|x|)\,d\mu(x)
    +\int\Psi(|y|)\,d\nu(y)\right].
\]
Thus $\gamma$ minimizes $\widehat c_2$ over all of $\Pi(\mu,\nu)$.
Put $M\coloneqq\mu(\mathbb R)=\nu(\mathbb R)$.  If $M=0$, the conclusion is
vacant.  If $M>0$, then $M^{-1}\gamma$ is an optimal probability plan for the
same nonnegative lower semicontinuous cost $\widehat c_2$, and its cost is
finite by the assumed finite secondary value and the finite marginal
$\Psi$-moments.  
We make use of the standard result characterizing optimality of
couplings in terms of so-called $c$-cyclic monotonicity
\cite[Theorem~3.2]{ambrosio2003existence}.  The necessity part of that
theorem, applied to the optimal probability plan $M^{-1}\gamma$,
therefore provides a $\widehat c_2$-cyclically monotone Borel set
$S\subset\Gamma_u$ of full $M^{-1}\gamma$-mass, hence of full
$\gamma$-mass.

For points in $S$, the 1-Lipschitz property of $u$ gives
\begin{align*}
    \sum_i|x_i-y_i|
    &=\sum_i\bigl(u(x_i)-u(y_i)\bigr)\\
    &=\sum_i\bigl(u(x_i)-u(y_{\tau(i)})\bigr)
    \le\sum_i|x_i-y_{\tau(i)}|.
\end{align*}
If equality holds, every nonnegative Lipschitz gap in the last inequality
vanishes, so all rearranged pairs $(x_i,y_{\tau(i)})$ belong to
$\Gamma_u$.  Cyclical monotonicity of $\widehat c_2$ can then be applied
to this permutation.  Its separable shift cancels because the source
list is unchanged and the target list is merely permuted.  The remaining
inequality is exactly the asserted inequality for $c_2$.
\end{proof}

\begin{theorem}[Logarithmic variational property of the excursion coupling]
\label{thm:power-ec-one-dimensional}
Let $\mu,\nu\in\mathcal M_{1+}(\mathbb R)$ be mutually singular finite
positive Borel measures with equal positive mass.  Assume that $\mu$ is
atomless and that $\nu$ stochastically dominates $\mu$.  Let
$\gamma_{\mathrm{EC}}$ be Juillet's excursion coupling and set
\[
    \mathcal A\coloneqq
    \{\gamma\in\Pi(\mu,\nu):
      \gamma(\{(x,y):x\le y\})=\mu(\mathbb R)\}.
\]
Then $\mathcal A=\Pi_1(\mu,\nu)$, and
$\gamma_{\mathrm{EC}}$ is the unique solution of
\[
    \min_{\gamma\in\mathcal A}J(\gamma)
    =
    \min_{\gamma\in\mathcal A}
    \int_{\mathbb R^2}-|x-y|\log|x-y|\,d\gamma(x,y).
\]
\end{theorem}

\begin{proof}
The log moment implies a finite first moment.  Hence
Theorem~\ref{thm:EC-map-gives-a-plan-1D} and
Lemma~\ref{lem:forward-equals-l1-ordered} give
$\gamma_{\mathrm{EC}}\in\mathcal A$ and
$\mathcal A=\Pi_1(\mu,\nu)$.  The class $\mathcal A$ is weakly compact:
the fixed-marginal class is compact, and the forward condition is closed
by the portmanteau theorem.  Lemma~\ref{lem:power-J-continuity} therefore
shows that $J$ has a finite minimizer on $\mathcal A$.
The function $g$ is continuous and strictly concave on $[0,\infty[$;
continuity at $r=0$ follows from $\lim_{r\to0^+}r\log r=0$, and strict concavity on $]0,\infty[$ follows
from $g''(r)=-1/r$.
Moreover,
\[
    g(\|x-y\|)\ge-\Psi(\|x-y\|)
    \ge-C_\Psi\bigl(1+\Psi(|x|)+\Psi(|y|)\bigr)
\]
by Lemma~\ref{lem:power-two-point-orlicz}.  Apply
Lemma~\ref{lem:power-orlicz-lexicographic} to any minimizer $\gamma$, with
$c_2(x,y)=g(|x-y|)$, and let $S$ be the resulting full-mass set.  Mutual
singularity lets us remove the diagonal from $S$, and the forward
condition then gives $x<y$ for every $(x,y)\in S$.

The set $S$ has the same primary-tie and secondary-domination property
that Lemma~\ref{lem:secondary-swap-general-profile} gives in the proof
of Theorem~\ref{thm:5.2'}.  So the same argument applies: crossing
arches and connected arches are excluded by
Lemma~\ref{lem:strict-concavity-arch-comparison}, and nested arches
automatically share an orientation because every pair in $S$ is
forward.  Hence $\gamma$ is concentrated on a full-mass set satisfying
Juillet's monotone-arch conditions, and $\gamma=\gamma_{\mathrm{EC}}$ by
Juillet's Main Theorem~\cite{juillet2020solution}.
Since $\gamma$ was an arbitrary minimizer, the minimizer is unique.

\end{proof}

\begin{theorem}[Logarithmic comparison along transport rays]
\label{thm:power-ray-construction}
Assume that $\mu,\nu,\gamma$, and $\Gamma$ satisfy the hypotheses of
Theorem~\ref{thm:6.1'}, and assume in addition that
$\mu,\nu\in\mathcal M_{1+}(\mathbb R^n)$.
Let $\gamma_\#^\Gamma$ be the plan constructed
in Theorem~\ref{thm:6.1'} for this fixed support $\Gamma$. Then
\[
    J(\gamma_\#^\Gamma)\le J(\gamma)\, ,
\]
and the inequality is strict unless $\gamma=\gamma_\#^\Gamma$.
\end{theorem}

\begin{proof}
Use the ray disintegration from the proof of Theorem~\ref{thm:6.1'}:
\[
    \gamma=\int\gamma_C\,d\sigma(C).
\]
The log-moment part of Lemma~\ref{lem:raywise-inheritance} shows that the
conditional marginals $\mu_C,\nu_C$ have finite
$\Psi(\|z\|)$ moments for $\sigma$-almost every ray, in addition to
inheriting mutual singularity, stochastic order, and atomlessness of the
conditional source.  If $s_C$ is an oriented coordinate on a fixed ray,
then $|s_C(z)|= \|z-z_C\| \le \|z\|+\|z_C\|$ for a fixed $z_C\in C$; the scaling estimate
from Lemma~\ref{lem:power-two-point-orlicz} therefore gives finite
$\Psi(|s_C|)$ moments as required by the one-dimensional theorem.
The measurable-gluing step in the proof of
Theorem~\ref{thm:6.1'} gives a measurable family of raywise EC plans
$\gamma_{C,\#}$ and
\[
    \gamma_\#^\Gamma=\int\gamma_{C,\#}\,d\sigma(C)\, .
\]
Theorem~\ref{thm:power-ec-one-dimensional}, after identifying each
oriented ray with an interval of $\mathbb R$, gives
\[
    \int g(\|x-y\|)\,d\gamma_{C,\#}
    \le
    \int g(\|x-y\|)\,d\gamma_C
\]
for $\sigma$-almost every $C$, with equality exactly when
$\gamma_C=\gamma_{C,\#}$.

Lemma~\ref{lem:power-fixed-marginal-ui} gives absolute integrability of
$g(\|x-y\|)$ for both global couplings $\gamma$ and
$\gamma_\#^\Gamma$, since they have the same fixed marginals.  It follows that we can disintegrate $J(\cdot)$ for both couplings as per
\begin{align*}
    J(\gamma)
    &=\int\left(\int g(\|x-y\|)\,d\gamma_C\right)d\sigma(C)\, ,\\
    J(\gamma_\#^\Gamma)
    &=\int\left(\int g(\|x-y\|)\,d\gamma_{C,\#}\right)d\sigma(C)\, .
\end{align*}
Integrating the raywise inequalities proves the comparison.  If equality
holds globally, the nonnegative, measurable raywise gaps have integral
zero.  They therefore vanish for $\sigma$-almost every ray, and raywise
uniqueness gives $\gamma_C=\gamma_{C,\#}$ almost everywhere.  Measurable
gluing then gives $\gamma=\gamma_\#^\Gamma$. 
\end{proof}

\begin{theorem}[Logarithmic secondary uniqueness]
\label{thm:power-secondary-uniqueness}
Let $\mu,\nu\in\mathcal M_{1+}(\mathbb R^n)$ be mutually singular finite
positive Borel measures with equal positive mass, and assume
$\mu\ll\mathcal L^n$.  Then
\[
    \min_{\gamma\in\Pi_1(\mu,\nu)}J(\gamma)
    =
    \min_{\gamma\in\Pi_1(\mu,\nu)}
    \int-\|x-y\|\log\|x-y\|\,d\gamma(x,y)
\]
has a finite value and the unique minimizer $\gamma_\#$ from
Theorem~\ref{thm:7.2'}. This minimizer is induced by the intrinsic generalized
EC map $t_\#$; in particular, it is the same plan selected by every admissible
strictly concave secondary profile.
\end{theorem}

\begin{proof}
The set $\Pi_1(\mu,\nu)$ is nonempty and weakly compact, and
Lemma~\ref{lem:power-J-continuity} shows that $J$ is finite and continuous
there.  Thus a minimizer exists and the minimum is finite.

Fix the 1-Lipschitz Kantorovich potential $u$ from
Lemma~\ref{lem:linear-cost-contact-set}, and set
\[
    \Delta\coloneqq\{(x,x):x\in\mathbb R^n\},\qquad
    \Gamma^*\coloneqq \Gamma_u\setminus\Delta.
\]
The set $\Gamma^*$ is $\sigma$-compact: it is the union, over
$j,k\ge1$, of the compact sets
\[
    \Gamma_u\cap
    \{(x,y):\|(x,y)\|\le j\}\cap
    \{(x,y):\|x-y\|\ge1/k\}.
\]
Theorem~\ref{thm:6.2'} verifies the geometric hypotheses of
Theorem~\ref{thm:6.1'} for this diagonal-free subset of $\Gamma_u$.

Every plan in $\Pi_1(\mu,\nu)$ is concentrated on $\Gamma_u$.  It also
gives no mass to $\Delta$: if $A$ is a Borel separator with
$\mu(A^c)=0$ and $\nu(A)=0$, then
\[
    \gamma(\Delta)
    \le\gamma(A^c\times\mathbb R^n)
      +\gamma(\mathbb R^n\times A)
    =0.
\]
Hence every secondary minimizer is concentrated on the same set
$\Gamma^*$.

Let $\gamma_\#^{\Gamma^*}$ be the plan supplied by
Theorem~\ref{thm:6.1'} for this fixed support, and let $\gamma$ be any
minimizer of $J$. Since $\Gamma^*$ is contained in the cyclically monotone
contact set $\Gamma_u$, conclusion (a) of Theorem~\ref{thm:6.1'} gives
$\gamma_\#^{\Gamma^*}\in\Pi_1(\mu,\nu)$. The logarithmic comparison gives
\[
    J(\gamma_\#^{\Gamma^*})\le J(\gamma).
\]
Minimality supplies the reverse inequality, so equality holds. The strict
equality statement in Theorem~\ref{thm:power-ray-construction} forces
$\gamma=\gamma_\#^{\Gamma^*}$. Thus the logarithmic minimizer is unique,
induced by a map, and agrees raywise with the one-dimensional excursion
coupling. The proof of Theorem~\ref{thm:7.2'} defines its intrinsic plan
$\gamma_\#$ using this same fixed-support construction
$\gamma_\#^{\Gamma^*}$; hence the logarithmic minimizer is precisely
$\gamma_\#$ and it is induced by the same map $t_\#$.
\end{proof}

\begin{theorem}[Convergence of power-cost optimizers to the generalized EC map.]
\label{thm:main-theorem-power-costs}
Let $\mu,\nu\in\mathcal M_{1+}(\mathbb R^n)$ be mutually singular finite
positive Borel measures with equal positive mass, and assume
$\mu\ll\mathcal L^n$.  For every $\epsilon\in]0,1[$, the problem with cost
\[
    c_\epsilon(x,y)=\|x-y\|^{1-\epsilon}
\]
has a unique optimal plan $\gamma_\epsilon$, induced by a transport map
$t_\epsilon$.  As $\epsilon\to0^+$,
\[
    \gamma_\epsilon\longrightarrow\gamma_\#
    \quad\text{weakly},
    \qquad
    t_\epsilon\longrightarrow t_\#
    \quad\text{in $\mu$-measure},
\]
where
$\gamma_\#=(\mathrm{Id},t_\#)_\#\mu$
is the unique logarithmic secondary minimizer from
Theorem~\ref{thm:power-secondary-uniqueness}.
\end{theorem}

\begin{proof}
For each $\epsilon\in]0,1[$, the profile
$r\mapsto r^{1-\epsilon}$ is nonnegative, increasing, and strictly
concave.  The concave-cost transport theorem of Gangbo--McCann, in the
form of \cite[Theorem~3.3]{pegon2013concave-general}, applied after
normalizing the common mass, gives the unique optimal plan and its map
representation.

The class $\Pi(\mu,\nu)$ is weakly compact.  Hence every sequence
$\epsilon_k\to0^+$ has a subsequence for which
$\gamma_{\epsilon_k}$ converges weakly.  By
Proposition~\ref{prop:power-gamma-selection}, every such limit minimizes
$J$ on $\Pi_1(\mu,\nu)$.  Theorem~\ref{thm:power-secondary-uniqueness}
identifies its unique minimizer as
$\gamma_\#$.  Thus every subsequential limit is
$\gamma_\#$.  The narrow topology on finite
measures of fixed mass over the Polish product space is metrizable.  If the
full family failed to converge, a sequence $\epsilon_k\to0^+$ would
remain outside a neighborhood of
$\gamma_\#$; compactness would produce a convergent
subsequence, contradicting the preceding identification.  Hence the full
family converges weakly to $\gamma_\#$.

It remains to pass from plans to maps.  Fix $\rho>0$.  By Lusin's theorem,
for every $\delta>0$ there is a compact set $K$ such that
$\mu(K^c)<\delta$ and $t_\#|_K$ is continuous.  The set
\[
    B_{\rho,K}\coloneqq 
    \{(x,y):x\in K,\
      \|y-t_\#(x)\|\ge\rho\}
\]
is closed and has $\gamma_\#$-mass zero.  The
portmanteau theorem gives
\[
    \limsup_{\epsilon\to0^+}\gamma_\epsilon(B_{\rho,K})=0\, .
\]
Since $\gamma_\epsilon=(\mathrm{Id},t_\epsilon)_\#\mu$,
\[
    \mu\{x:
      \|t_\epsilon(x)-t_\#(x)\|\ge\rho\}
    \le\mu(K^c)+\gamma_\epsilon(B_{\rho,K})\, .
\]
Taking the upper limit and then letting $\delta\to0$ proves convergence
in $\mu$-measure.
\end{proof}

\bibliographystyle{plainnat}
\bibliography{ref}

\end{document}